\documentclass[preprint,12pt]{elsarticle}

\usepackage{subfigure}
\usepackage{stmaryrd}
\usepackage{amssymb}
\usepackage{verbatim}
\usepackage{mathrsfs}
\usepackage{dsfont}
\usepackage[centertags]{amsmath}
\usepackage{amsfonts}
\usepackage{amssymb}
\usepackage{amsthm}
\usepackage{newlfont}
\usepackage{color}
\usepackage{soul}
\usepackage{bm}
\usepackage[top=1.00 in,bottom=1.00 in,left=1.00 in,right=1.00 in]{geometry}
\def\3bar{{|\hspace{-.02in}|\hspace{-.02in}|}}
\hfuzz=\maxdimen
\newtheorem{theorem}{Theorem}[section]

\newtheorem{lemma}[theorem]{Lemma}

\theoremstyle{definition}

\theoremstyle{remark}

\numberwithin{equation}{section}

\begin{document}

\begin{frontmatter}



\title{Weak Galerkin finite element method for linear poroelasticity problems}

\author[label1,label2]{Shanshan Gu}
\author[label1]{Shimin Chai}
\ead{chaism@jlu.edu.cn }
\author[label4]{Chenguang Zhou\corref{cor1}}
\ead{zhoucg@bjut.edu.cn }
\cortext[cor1]{Corresponding author}
\address[label1]{School of Mathematics, Jilin University, Changchun, Jilin 130012, China.}
\address[label2]{State Key Laboratory of Polymer Physics and Chemistry, Changchun Institute of Applied Chemistry, Chinese Academy of Sciences, Changchun 13022, China.}
\address[label4]{Faculty of Science, Beijing University of Technology, Beijing 100124, China.}

\begin{abstract}
This paper is devoted to a weak Galerkin (WG) finite element method for linear poroelasticity problems where weakly defined divergence and gradient operators over discontinuous functions are introduced. We establish both the continuous and discrete time WG schemes, and obtain their optimal convergence order estimates in a discrete $H^1$ norm for the displacement and in an $H^1$ type and $L^2$ norms for the pressure. Finally, numerical experiments are presented to illustrate the theoretical error results in different kinds of meshes which shows the WG flexibility for mesh selections, and to verify the locking-free property of our proposed method.\\

\vspace{-1 ex}
\noindent $\textrm{\bf{Mathematics Subject Classification 2020}}$: 65M60, 65M15, 76S05
\end{abstract}

\begin{keyword}


Weak Galerkin \sep finite element method \sep linear poroelasticity problem \sep optimal pressure error estimate \sep locking-free property
\end{keyword}

\end{frontmatter}


\section{Introduction}\label{intro}

\indent In this paper, we consider the following two-field Navier-formed Biot's consolidation model which depicts a quasi-static flow in a saturated deformable poroelastic medium. Let $\Omega$ be a convex polygonal or polyhedral domain in $\mathbb{R}^d$ $(d=2, 3)$ with smooth boundary $\Gamma=\Gamma_{p, D}\cup\Gamma_{p, N}$, where $\Gamma_{p, D}$ is nonempty and $T$ is a final time. The displacement of porous solid media $\bm{u}(t):\Omega\to \mathbb{R}^d$ and the pore pressure of fluid $p(t):\Omega\to\mathbb{R}$ satisfy the system
\begin{align}
-(\lambda+\mu)\nabla(\nabla\cdot\bm{u})-\mu\triangle\bm{u}+\alpha\nabla p &= \bm{f}, \quad \textrm{in} \ \Omega, \ t\in (0, T],\label{eq:5}\\
\frac{\partial}{\partial t}(c_0 p+\alpha\nabla\cdot\bm{u})-\nabla\cdot(\kappa\nabla p) &= g, \quad \textrm{in} \ \Omega, \ t\in (0, T],\label{eq:6}
\end{align}
with the boundary conditions
\begin{align}
\bm{u} &= \bm{0}, \quad \textrm{on} \ \Gamma,\label{DirichletU}\\
p &= 0, \quad \textrm{on} \ \Gamma_{p, D},\label{DirichletP}\\
\kappa\nabla p\cdot\bm{n} &= \gamma, \quad \textrm{on} \ \Gamma_{p, N},\notag
\end{align}
and the initial conditions
\begin{align}
\bm{u}(\cdot, 0) &= \bm{u}^0,\quad \textrm{in} \ \Omega, \label{eq:1}\\
p(\cdot, 0) &= p^0,\quad \textrm{in} \ \Omega. \label{eq:2}
\end{align}
Here $\bm{f}(t):\Omega\to\mathbb{R}^d$ is the body force, $g(t):\Omega\to\mathbb{R}$ is the volumetric fluid source (or sink), $\bm{\beta}(t):\Omega\to\mathbb{R}^d$ represents the prescribed surface traction, and $\gamma(t):\Omega\to\mathbb{R}$ states the prescribed discharge on the boundary. $\varepsilon(\bm{u})=\frac{1}{2}(\nabla\bm{u}+\nabla\bm{u}^{\textrm{T}})$ stands for the strain tensor, $\lambda$ and $\mu$ are the Lam\'{e} constants, $\alpha$ is the Biot-Willis parameter, $c_0 \geq 0$ represents the constrained specific storage coefficient, $\kappa$ denotes the hydraulic conductivity. $I$ states the identity tensor, and $\bm{n}$ is the unit outward normal vector. In this paper, for simplicity, we assume that $\alpha = 1$. If $\alpha \neq 1$, one may reduce the model to the one with $\alpha = 1$ by rescaling the equations.\\

\indent As one of the classic poroelasticity problems, the Biot's consolidation model has been treated by various numerical methods, such as finite element methods, finite difference methods, hybrid discontinuous Galerkin (HDG) methods, hybrid high-order (HHO) methods and WG methods. In \cite{MR3874793}, the authors present a finite element discretization preserving pointwise mass balance for the Biot's consolidation model and provide numerical results to demonstrate the method. In \cite{MR4153035}, Chen and Yang design a three-variable weak form with mixed finite element and derive optimal convergence order estimates. Stability estimates and convergence analysis of finite difference methods for the Biot's consolidation model are presented in \cite{MR1957690}. In \cite{MR2371348}, the authors deal with the numerical solution of a secondary consolidation Biot model, and a family of finite difference methods on staggered grids in both time and spatial variables is considered. Fu employs high order HDG methods for the Biot' consolidation equations and backward Euler methods are used for temporal discretization \cite{MR3907413}. In \cite{MR3504993}, the authors discretize the displacements describing the elastic deformation by the HHO method \cite{MR3283758}, and the pressure representing the flow problem by the symmetric weighted interior penalty discontinuous Galerkin method \cite{MR2383212}. In \cite{MR3775100}, the authors adopt WG linear finite elements for spatial discretzation and backward Euler scheme for temporal discretization in order to obtain an implicit fully discretized scheme of the Biot's consolidation model. The results of \cite{MR3775100} are generalized to high order elements in \cite{ar53}, and degrees of freedom are reduced on element boundaries for the pressure approximation without compromising the accuracy. In \cite{MR3858071}, the authors apply a modified WG method to the Biot's problem and derive the error estimates of semi-discrete and fully discrete schemes.\\

\indent Generally, WG methods, first proposed and analyzed in \cite{article4}, refers to finite element techniques for solving partial differential equations where differential operators (e.g., gradient, divergence, curl, etc.) are approximated by weak forms as distributions. Later WG is successfully extended to elliptic interface problems \cite{article6, MR3546841}, Helmholtz equations \cite{ar55, ar54}, linear parabolic equations \cite{article12, article31, article32}, and further developed for other applications, such as biharmonic problems \cite{article36, MR4048637}, Stokes problems \cite{MR4399131, ar58}, Stokes-Darcy problems \cite{MR4316143, MR4126527} and stochastic partial differential equations \cite{MR3858074, MR4021050}, etc. The idea of parameter free stabilization term is introduced in \cite{article7} to improve the flexibility of element construction and mesh generation. The resulting WG method is no longer limited to the RT \cite{article10} or BDM \cite{article18} elements in the computation of discrete weak gradient. In analogy with discrete weak gradient, discrete weak divergence is introduced in \cite{article11} where the proposed weak Galerkin mixed finite element method (WGMFEM) is applicable for general finite element partitions consisting of shape regular polygons in 2D or polyhedra in 3D. Additionally, WGMFEM is developed in second-order elliptic equations with Robin boundary conditions \cite{article41}, heat equations \cite{article23, MR4062237}, Helmholtz equations with large wave numbers \cite{ar62}, and quasi-linear poroelasticity problems \cite{MR4136610}, etc.\\

\indent To our best knowledge, up to now, there have been two papers in total, i.e., \cite{MR3775100, ar53}, studying the linear two-field Biot model from the point of WG discretization. In \cite{MR3775100, ar53}, the authors adopt the piecewise polynomials with the same degree to discretize the displacement and pressure in the interior of elements, which causes that they only obtain the suboptimal error convergence rates for the pressure theoretically. In addition, the authors of both papers consider only the case of spurious pressure oscillations in the numerical experiments, not the locking problem.\\

\indent Based on above, in this paper, we propose a WG method for the linear two-field (displacement and pressure) Biot's consolidation model in the Navier form and set up the continuous and discrete time WG schemes. We respectively design $[P_{j+1}]^d$-$[P_j]^d$ and $P_j$-$P_{j-1}$ $(d=2, 3 \textrm{ and } j \geq 1)$ WG combinations to gain the displacement and pressure approximations. With the use of these combinations satisfying the discrete$\textrm{ inf-sup }$condition, we derive the optimal order error estimates of the semi-discrete and fully discrete schemes in a discrete $H^1$ norm for the displacement and in an $H^1$ type and $L^2$ norms for the pressure. Finally, some numerical examples are supplied to illustrate the advantages of our proposed method from the two aspects, the good mesh flexibility and locking-free property for the system \eqref{eq:5}-\eqref{eq:2}.\\

\indent The outline of this paper goes as follows. In Section \ref{MathematicalModelAndItsVariationalFormulation}, we establish the weak formulation based on some necessary notations and definitions. The specific WG method is introduced in Section \ref{WeakGalerkinMethod} and we provide the semi-discrete and fully discrete numerical schemes. In Section \ref{ErrorAnalysis}, the optimal order convergence estimates of two numerical schemes are derived, and ultimately in Section \ref{NumericalExperiments}, we supply numerical experiments to validate our theoretical findings and expectation.\\

\section{Notations and variational formulation}\label{MathematicalModelAndItsVariationalFormulation}

\indent In this section, before bringing in the variational formulation of \eqref{eq:5} and \eqref{eq:6}, we firstly present some useful notations and definitions. In this paper, we utilize the standard definition of Sobolev space $H^{s}(\Omega)$ with $s\geq 0$ (cf. \cite{MR0450957}). The associated inner-product and norm in $H^s(\Omega)$ are denoted by $(\cdot, \cdot)_s$ and $\|\cdot\|_s$, respectively. When $s=0$, $H^0(\Omega)$ coincides with the space of square-integrable functions $L^2(\Omega)$. In this case, the subscript $s$ is suppressed from the notation of inner product and norm. The above notations and definitions can easily be extended to vector-valued functions. The inner-product and norm for such functions shall follow the same naming convention. We also define two spaces
\begin{equation*}
\bm{H}^1_{\bm{u}}(\Omega) := \{\bm{v} \in [H^1(\Omega)]^d,\ \bm{v}=\bm{0}\ \textrm{on } \Gamma\},
\end{equation*}
and
\begin{equation*}
H^1_{p}(\Omega) := \{q \in H^1(\Omega),\ q=0\ \textrm{on } \Gamma_{p, D}\}.
\end{equation*}
In addition, the letter $C$ (with or without subscripts) denotes a generic positive constant which may be different at its different occurrences throughout this paper.\\

\indent Now, we can define the variational equations of \eqref{eq:5} and \eqref{eq:6} as follows: For any $t\in(0, T]$, seek $\bm{u}(t)\in \bm{H}^1_{\bm{u}}(\Omega)$ and $p(t)\in H^1_p(\Omega)$ such that
\begin{align*}
(\lambda+\mu)(\nabla\cdot\bm{u}, \nabla\cdot\bm{v})+\mu(\nabla\bm{u}, \nabla\bm{v})-(\nabla\cdot\bm{v}, p) &= (\bm{f}, \bm{v}), \quad \forall\ \bm{v}\in \bm{H}^1_{\bm{u}}(\Omega),\\
(\frac{\partial}{\partial t}(c_0 p+\nabla\cdot\bm{u}), q)+(\kappa\nabla p, \nabla q) &= (g, q)+\langle \gamma, q\rangle_{\Gamma_{p, N}}, \quad \forall\ q \in H_p^1(\Omega),
\end{align*}
with the initial conditions \eqref{eq:1} and \eqref{eq:2}.\\

\section{WG method}\label{WeakGalerkinMethod}

\indent In this section, the definitions of discrete weak divergence and weak gradient operators are firstly rendered. The key to the WG method is to use discrete weak differential operators in place of standard differential operators in the variational form of the original system \eqref{eq:5} and \eqref{eq:6}. Then we supply the semi-discrete and fully discrete WG schemes used for our error analysis and numerical computation.\\

\indent Let $\mathcal{T}_h$ be a finite element partition of the domain $\Omega\subset\mathbb{R}^d$ consisting of polygons $(d=2)$ or polyhedra $(d=3)$ satisfying the shape regularity requirements A1-A4 in \cite{article11}. Denote by $h_K$ the partition diameter of the element $K\in\mathcal{T}_h$ with boundary $\partial K$, and $h=\max\limits_{K}h_K$. Let $\mathcal{E}_h$ be the set of all edges or faces in $\mathcal{T}_h$, and $\mathcal{E}_h^0=\mathcal{E}_h\backslash\Gamma$ be the set of all interior edges or faces. The sets of polynomials with degree no more than $j$ on each $K$ and $e\in \mathcal{E}_h$ are denoted by $P_j(K)$ and $P_j(e)$, respectively.\\

\indent For the displacement $\bm{u}$, we define two weak vector-valued finite element spaces as, for any integer $j \geq 1$,
\begin{equation*}
\bm{V}_h := \{\bm{v}_{h}=\{\bm{v}_0, \bm{v}_b\}: \{\bm{v}_0, \bm{v}_b\}|_K\in[P_{j+1}(K)]^d \times [P_j(e)]^d, \ K\in\mathcal{T}_h,\ e\subset \partial K\},
\end{equation*}
and
\begin{equation*}
\bm{V}_h^{\bm{u}} := \{\bm{v}_{h}=\{\bm{v}_0, \bm{v}_b\}\in \bm{V}_h: \bm{v}_b=\bm{0}\ \textrm{on } \Gamma\}.
\end{equation*}
Based on the above definitions of spaces, we bring in discrete weak divergence and weak gradient operators. For $\bm{v}_h\in \bm{V}_h$, define $\nabla_{w,K}\cdot\bm{v}_h\in P_j(K)$ and $\nabla_{w,K}\bm{v}_h\in[P_j(K)]^{d\times d}$ on each element $K$ as follows,
\begin{align*}
(\nabla_{w,K}\cdot\bm{v}_h, \psi)_K &= -(\bm{v}_0, \nabla\psi)_K+\langle\bm{v}_b\cdot\bm{n}, \psi\rangle_{\partial K}, \quad \forall\ \psi\in P_j(K),\\
(\nabla_{w,K}\bm{v}_h, \phi)_K &= -(\bm{v}_0, \nabla\cdot\phi)_K+\langle\bm{v}_b, \phi\cdot\bm{n}\rangle_{\partial K}, \quad \forall\ \phi\in[P_j(K)]^{d\times d}.
\end{align*}
Then define the global weak divergence and weak gradient by patching the local ones, i.e.,
\begin{align*}
(\nabla_w\cdot\bm{v}_h)|_K &= \nabla_{w,K}\cdot(\bm{v}_h|_K), \quad \forall\ \bm{v}_h\in \bm{V}_h,\\
(\nabla_w\bm{v}_h)|_K &= \nabla_{w,K}(\bm{v}_h|_K), \quad \forall\ \bm{v}_h\in \bm{V}_h.
\end{align*}

\indent Similarly, for the pressure $p$, we introduce its weak finite element spaces and discrete weak gradient operator. For any integer $j \geq 1$, define
\begin{equation*}
W_h := \{q_{h}=\{q_0, q_b\}: \{q_0, q_b\}|_K\in P_{j}(K) \times P_{j-1}(e), \ K\in\mathcal{T}_h,\ e\subset \partial K\},
\end{equation*}
and
\begin{equation*}
W_h^{p} := \{q_{h}=\{q_0, q_b\}\in W_h: q_b=0\ \textrm{on } \Gamma_{p,D}\}.
\end{equation*}
For $q_h\in W_h$, define $\nabla_{w,K}q_h\in[P_{j-1}(K)]^d$ on each element $K$ as follows,
\begin{equation*}
(\nabla_{w,K}q_h, \bm{\zeta})_K = -(q_0, \nabla\cdot\bm{\zeta})_K+\langle q_b, \bm{\zeta}\cdot\bm{n}\rangle_{\partial K}, \quad \forall\ \bm{\zeta}\in[P_{j-1}(K)]^d.
\end{equation*}
Then define the global weak gradient by patching the local ones, i.e.,
\begin{equation*}
(\nabla_w q_h)|_K = \nabla_{w,K}(q_h|_K), \quad \forall\ q_h\in W_h.
\end{equation*}

\indent Before numerical schemes, we first present several definitions of $L^2$ projection operator. For each $K \in \mathcal{T}_h$, denote by $\bm{Q}_0$ the $L^2$ projection operator from $[L^2(K)]^d$ onto $[P_{j+1}(K)]^d$, by $Q_0$ the $L^2$ projection operator from $L^2(K)$ onto $P_j(K)$. For each $e\in\mathcal{E}_h$, denote by $\bm{Q}_b$ the $L^2$ projection operator from $[L^2(e)]^d$ onto $[P_j(e)]^d$, by $Q_b$ the $L^2$ projection operator from $L^2(e)$ onto $P_{j-1}(e)$. We shall combine $\bm{Q}_0$ with $\bm{Q}_b$, and $Q_0$ with $Q_b$, by writing $\bm{Q}_h=\{\bm{Q}_0, \bm{Q}_b\}$ and $Q_h=\{Q_0, Q_b\}$, respectively.\\

\indent Next we introduce several bilinear forms as follows: For $\bm{v}_h=\{\bm{v}_0, \bm{v}_b\}\in\bm{V}_h$, $\bm{w}_h=\{\bm{w}_0, \bm{w}_b\}\in \bm{V}_h$, $q_h=\{q_0, q_b\}\in W_h$, $\eta_h=\{\eta_0, \eta_b\}\in W_h$,
\begin{align*}
s_{\bm{u}}(\bm{v}_h,\bm{w}_h) &= \sum_{K\in\mathcal{T}_h}h_K^{-1}\langle \bm{Q}_b\bm{v}_0-\bm{v}_b, \bm{Q}_b\bm{w}_0-\bm{w}_b\rangle_{\partial K},\\
s_p(q_h,\eta_h) &= \sum_{K\in\mathcal{T}_h}h_K^{-1}\langle Q_bq_0-q_b, Q_b\eta_0-\eta_b\rangle_{\partial K},\\
a_{\bm{u}}(\bm{v}_h, \bm{w}_h) &= \sum_{K\in\mathcal{T}_h}(\lambda+\mu)(\nabla_{w}\cdot\bm{v}_h, \nabla_w\cdot\bm{w}_h)_K+\sum_{K\in\mathcal{T}_h}\mu(\nabla_w\bm{v}_h, \nabla_w\bm{w}_h)_K+s_{\bm{u}}(\bm{v}_h, \bm{w}_h),\\
a_p(q_h, \eta_h) &= \sum_{K\in\mathcal{T}_h}\kappa(\nabla_wq_h, \nabla_w\eta_h)_K+s_p(q_h, \eta_h),\\
b(\bm{v}_h, q_h) &= \sum_{K\in\mathcal{T}_h}(\nabla_w\cdot\bm{v}_h, q_0)_K.
\end{align*}
We also define two norms for $\bm{v}_h\in\bm{V}_h^{\bm{u}}$ and $q_h\in W_h^p$ by
\begin{align*}
|\!|\!|\bm{v}_h|\!|\!|_{\bm{V}} &:= \{ a_{\bm{u}}(\bm{v}_h, \bm{v}_h)\}^{\frac{1}{2}},\\
|\!|\!| q_h|\!|\!|_W &:= \{ a_p(q_h, q_h)\}^{\frac{1}{2}}.
\end{align*}
From \cite{MR3858071}, we know that the bilinear forms $a_{\bm{u}}(\cdot, \cdot)$ and $a_p(\cdot, \cdot)$ are bounded, symmetric and coercive in $\bm{V}_h^{\bm{u}}$ and $W_h^p$, respectively. And from \cite{article38}, the bilinear form $b(\cdot, \cdot)$ is bounded in $\bm{V}_h^{\bm{u}}\times W_h^p$ and satisfies the inf-sup condition.\\

\indent Now we can formulate the semi-discrete WG scheme. For any $t\in(0,T]$, find $\bm{u}_h=\{\bm{u}_0,\bm{u}_b\}\in \bm{V}_h^{\bm{u}}$ and $p_h=\{ p_0, p_b\}\in W_h^p$ such that
\begin{align}
&a_{\bm{u}}(\bm{u}_h, \bm{v}_h)-b(\bm{v}_h, p_h) = (\bm{f}, \bm{v}_0), \quad \forall\ \bm{v}_h=\{\bm{v}_0, \bm{v}_b\}\in\bm{V}_h^{\bm{u}},\label{eq:3}\\
&(c_0p_{0,t},q_0)+b(\bm{u}_{h,t},q_h)+a_p(p_h,q_h) = (g,q_0)+\langle \gamma, q_b\rangle_{\Gamma_{p, N}}, \quad \forall\ q_h=\{q_0, q_b\}\in W_h^p,\label{eq:4}
\end{align}
where $(c_0p_{0,t},q_0)=\sum\limits_{K\in\mathcal{T}_h}(c_0p_{0,t},q_0)_K$. According to the properties of the bilinear forms $a_{\bm{u}}(\cdot, \cdot)$, $a_p(\cdot, \cdot)$ and $b(\cdot, \cdot)$, the solution of problem \eqref{eq:3} and \eqref{eq:4} exists and is unique.\\

\indent We turn our attention to the fully discrete numerical scheme. We introduce a time step size $\tau=\frac{T}{N}$ for some positive integer $N$ and $t_n=n\tau$ for $n = 0, 1, ..., N$. By $\bm{u}_h^n$ and $p_h^n$, we denote the approximation of $\bm{u}(t_n)$ and $p(t_n)$, respectively. We use the backward Euler method to approximate the time derivative in \eqref{eq:4}, and then the fully discrete scheme reads: For $n = 1, ..., N$, seek $\bm{u}_h^n=\{ \bm{u}_0^n, \bm{u}_b^n\}\in\bm{V}_h^{\bm{u}}$ and $p_h^n=\{p_0^n, p_b^n\}\in W_h^p$ such that
\begin{align}
&a_{\bm{u}}(\bm{u}_h^n, \bm{v}_h)-b(\bm{v}_h, p_h^n) = (\bm{f}(t_n), \bm{v}_0), \ \forall\ \bm{v}_h=\{\bm{v}_0, \bm{v}_b\}\in\bm{V}_h^{\bm{u}},\label{eq:21}\\
&(c_0\partial_\tau p_0^n, q_0)+b(\partial_\tau \bm{u}_h^n, q_h)+a_p(p_h^n, q_h) = (g(t_n), q_0)+\langle \gamma(t_n), q_b\rangle_{\Gamma_{p, N}}, \ \forall\ q_h=\{q_0, q_b\}\in W_h^p,\label{eq:22}
\end{align}
where $\partial_\tau p_0^n=\frac{p_0^{n}-p_0^{n-1}}{\tau}$ and $\partial_\tau \bm{u}_h^n=\frac{\bm{u}_h^n-\bm{u}_h^{n-1}}{\tau}$.\\

\section{Error analysis}\label{ErrorAnalysis}

\indent In this section, we shall derive the optimal order error estimates for both continuous and discrete time WG methods.\\

\subsection{Continuous time WG method}

\indent Firstly, we bring in two useful $L^2$ projection operators. In addition to the projection operators $\bm{Q}_h=\{\bm{Q}_0, \bm{Q}_b\}$ and $Q_h=\{Q_0, Q_b\}$ mentioned above, for each element $K\in \mathcal{T}_h$, let $\mathbb{Q}_0$ and $\bm{\widehat{\mathcal{Q}}}_0$ be two local $L^2$ projection operators onto $[P_{j-1}(K)]^d$ and $[P_j(K)]^{d\times d}$, respectively. Then we have the following lemma.
\begin{lemma}\cite{article38, MR3775100}\label{Lemma1}
For any $\bm{v}\in [H^1(\Omega)]^d$ and $q\in H^1(\Omega)$, we have the following commutative properties of projection operators.
\begin{align*}
\nabla_{w}\cdot(\bm{Q}_h\bm{v}) &= Q_0(\nabla\cdot\bm{v}),\\
\nabla_{w}(\bm{Q}_h\bm{v}) &= \bm{\widehat{\mathcal{Q}}}_0(\nabla\bm{v}),\\
\nabla_{w}(Q_hq) &= \mathbb{Q}_0(\nabla q).
\end{align*}
\end{lemma}

\indent Based on these projection operators, the following results are presented as a preparation of error analysis.
\begin{lemma}\label{Lemma2}
The solution $\bm{u}$ and $p$ to the model problem \eqref{eq:5} and \eqref{eq:6} satisfies
\begin{align*}
&a_{\bm{u}}(\bm{Q}_h\bm{u}, \bm{v}_h)-b(\bm{v}_h, Q_hp) = (\bm{f}, \bm{v}_0)+s_{\bm{u}}(\bm{Q}_h\bm{u},\bm{v}_h)+l_1(\bm{u}, \bm{v}_h)+l_2(\bm{u},\bm{v}_h)-l_3(p, \bm{v}_h),\\
&(c_0Q_0p_t, q_0)+b(\bm{Q}_h\bm{u}_t, q_h)+a_p(Q_hp,q_h) = (g,q_0)+\langle \gamma, q_b\rangle_{\Gamma_{p,N}}+s_p(Q_hp,q_h)+l_4(p,q_h),
\end{align*}
for all $\bm{v}_h\in \bm{V}_h^{\bm{u}}$ and $q_h\in W_h^p$, where the linear functions $l_1$, $l_2$, $l_3$ and $l_4$ are defined as
\begin{align*}
l_1(\bm{u}, \bm{v}_h) &= \sum\limits_{K\in\mathcal{T}_h}(\lambda+\mu)\langle \nabla\cdot\bm{u}-Q_0(\nabla\cdot\bm{u}), (\bm{v}_0-\bm{v}_b)\cdot\bm{n}\rangle_{\partial K},\\
l_2(\bm{u}, \bm{v}_h) &= \sum\limits_{K\in\mathcal{T}_h}\mu\langle (\nabla\bm{u}-\bm{\widehat{\mathcal{Q}}}_0(\nabla\bm{u}))\cdot\bm{n}, \bm{v}_0-\bm{v}_b\rangle_{\partial K},\\
l_3(p, \bm{v}_h) &= \sum\limits_{K\in\mathcal{T}_h}\langle p-Q_0p, (\bm{v}_0-\bm{v}_b)\cdot\bm{n}\rangle_{\partial K},\\
l_4(p, q_h) &= \sum\limits_{K\in\mathcal{T}_h}\langle \kappa(\nabla p-\mathbb{Q}_0(\nabla p))\cdot\bm{n}, q_0-q_b\rangle_{\partial K}.
\end{align*}
\end{lemma}
\begin{proof}
For $\bm{v}_h\in\bm{V}_h^{\bm{u}}$, together with Lemma \ref{Lemma1}, the definition of discrete weak divergence, integration by parts and the definition of $Q_0$, we acquire
\begin{align*}
&\sum\limits_{K\in\mathcal{T}_h}(\lambda+\mu)(\nabla_w\cdot(\bm{Q}_h\bm{u}), \nabla_w\cdot\bm{v}_h)_K\\
= &\sum\limits_{K\in\mathcal{T}_h}(\lambda+\mu)(Q_0(\nabla\cdot\bm{u}), \nabla_w\cdot\bm{v}_h)_K\\
= &\sum\limits_{K\in\mathcal{T}_h}\{-(\lambda+\mu)(\nabla(Q_0(\nabla\cdot\bm{u})), \bm{v}_0)_K+(\lambda+\mu)\langle Q_0(\nabla\cdot\bm{u}), \bm{v}_b\cdot\bm{n}\rangle_{\partial K}\}\\
= &\sum\limits_{K\in\mathcal{T}_h}\{(\lambda+\mu)(Q_0(\nabla\cdot\bm{u}), \nabla\cdot\bm{v}_0)_K-(\lambda+\mu)\langle Q_0(\nabla\cdot\bm{u}), (\bm{v}_0-\bm{v}_b)\cdot\bm{n}\rangle_{\partial K}\}\\
= &\sum\limits_{K\in\mathcal{T}_h}\{ (\lambda+\mu)(\nabla\cdot\bm{u}, \nabla\cdot\bm{v}_0)_K-(\lambda+\mu)\langle Q_0(\nabla\cdot\bm{u}), (\bm{v}_0-\bm{v}_b)\cdot\bm{n}\rangle_{\partial K}\},
\end{align*}
which implies that
\begin{equation}\label{eq:7}
\sum\limits_{K\in\mathcal{T}_h}(\lambda+\mu)(\nabla\cdot\bm{u}, \nabla\cdot\bm{v}_0)_K = \sum\limits_{K\in\mathcal{T}_h}(\lambda+\mu)(\nabla_w\cdot(\bm{Q}_h\bm{u}), \nabla_w\cdot\bm{v}_h)_K+\sum\limits_{K\in\mathcal{T}_h}(\lambda+\mu)\langle Q_0(\nabla\cdot\bm{u}), (\bm{v}_0-\bm{v}_b)\cdot\bm{n}\rangle_{\partial K}.
\end{equation}
According to Lemma \ref{Lemma1}, the definition of discrete weak gradient, integration by parts and the definition of $\bm{\widehat{\mathcal{Q}}}_0$, it follows that
\begin{align*}
\sum\limits_{K\in\mathcal{T}_h}\mu(\nabla_w(\bm{Q}_h\bm{u}), \nabla_w\bm{v}_h)_K &= \sum\limits_{K\in\mathcal{T}_h}\mu(\bm{\widehat{\mathcal{Q}}}_0(\nabla\bm{u}), \nabla_w\bm{v}_h)_K\\
&= \sum\limits_{K\in\mathcal{T}_h}\{-\mu(\nabla\cdot(\bm{\widehat{\mathcal{Q}}}_0(\nabla\bm{u})), \bm{v}_0)_K+\mu\langle \bm{\widehat{\mathcal{Q}}}_0(\nabla\bm{u})\cdot\bm{n}, \bm{v}_b\rangle_{\partial K}\}\\
&= \sum\limits_{K\in\mathcal{T}_h}\{\mu(\bm{\widehat{\mathcal{Q}}}_0(\nabla\bm{u}), \nabla\bm{v}_0)_K-\mu\langle \bm{\widehat{\mathcal{Q}}}_0(\nabla\bm{u})\cdot\bm{n}, \bm{v}_0-\bm{v}_b\rangle_{\partial K}\}\\
&= \sum\limits_{K\in\mathcal{T}_h}\{\mu(\nabla\bm{u}, \nabla\bm{v}_0)_K-\mu\langle \bm{\widehat{\mathcal{Q}}}_0(\nabla\bm{u})\cdot\bm{n}, \bm{v}_0-\bm{v}_b\rangle_{\partial K}\},
\end{align*}
which shows that
\begin{equation}\label{eq:8}
\sum\limits_{K\in\mathcal{T}_h}\mu(\nabla\bm{u}, \nabla\bm{v}_0)_K = \sum\limits_{K\in\mathcal{T}_h}\mu(\nabla_w(\bm{Q}_h\bm{u}), \nabla_w\bm{v}_h)_K+\sum\limits_{K\in\mathcal{T}_h}\mu\langle \bm{\widehat{\mathcal{Q}}}_0(\nabla\bm{u})\cdot\bm{n}, \bm{v}_0-\bm{v}_b\rangle_{\partial K}.
\end{equation}
Because of the definition of discrete weak divergence, integration by parts and the definition of $Q_0$, we have
\begin{align*}
\sum\limits_{K\in\mathcal{T}_h}(\nabla_w\cdot\bm{v}_h, Q_0p)_K &= \sum\limits_{K\in\mathcal{T}_h}\{-(\bm{v}_0, \nabla(Q_0p))_K+\langle \bm{v}_b\cdot\bm{n}, Q_0p\rangle_{\partial K}\}\\
&= \sum\limits_{K\in\mathcal{T}_h}\{(\nabla\cdot\bm{v}_0, Q_0p)_K-\langle (\bm{v}_0-\bm{v}_b)\cdot\bm{n}, Q_0p\rangle_{\partial K}\}\\
&= \sum\limits_{K\in\mathcal{T}_h}\{(\nabla\cdot\bm{v}_0, p)_K-\langle (\bm{v}_0-\bm{v}_b)\cdot\bm{n}, Q_0p\rangle_{\partial K}\},
\end{align*}
which leads to
\begin{equation}\label{eq:9}
\sum\limits_{K\in\mathcal{T}_h}(\nabla\cdot\bm{v}_0, p)_K = \sum\limits_{K\in\mathcal{T}_h}(\nabla_w\cdot\bm{v}_h, Q_0p)_K+\sum\limits_{K\in\mathcal{T}_h}\langle (\bm{v}_0-\bm{v}_b)\cdot\bm{n}, Q_0p\rangle_{\partial K}.
\end{equation}
Now testing \eqref{eq:5} with $\bm{v}_0$ in $\bm{v}_h=\{\bm{v}_0, \bm{v}_b\}$ and using integration by parts, we obtain
\begin{equation}\label{eq:10}
\begin{split}
&\sum\limits_{K\in\mathcal{T}_h}(\lambda+\mu)(\nabla\cdot\bm{u}, \nabla\cdot\bm{v}_0)_K-\sum\limits_{K\in\mathcal{T}_h}(\lambda+\mu)\langle \nabla\cdot\bm{u}, \bm{v}_0\cdot\bm{n}\rangle_{\partial K}+\sum\limits_{K\in\mathcal{T}_h}\mu(\nabla\bm{u},\nabla\bm{v}_0)_K\\
&-\sum\limits_{K\in\mathcal{T}_h}\mu\langle \nabla\bm{u}\cdot\bm{n}, \bm{v}_0\rangle_{\partial K}-\sum\limits_{K\in\mathcal{T}_h}(p,\nabla\cdot\bm{v}_0)_K+\sum\limits_{K\in\mathcal{T}_h}\langle p, \bm{v}_0\cdot\bm{n}\rangle_{\partial K} = (\bm{f}, \bm{v}_0).
\end{split}
\end{equation}
Substituting \eqref{eq:7}, \eqref{eq:8} and \eqref{eq:9} into \eqref{eq:10} and adding $s_{\bm{u}}(\bm{Q}_h\bm{u}, \bm{v}_h)$ to the both sides, together with the boundary conditions, we present the first equality of Lemma \ref{Lemma2}.

\indent Next, we derive the other equation of this lemma. Considering the definition of $Q_0$ and Lemma \ref{Lemma1}, we find
\begin{equation*}
(c_0Q_0p_t, q_0) = (c_0p_t, q_0),
\end{equation*}
and
\begin{equation*}
\sum\limits_{K\in\mathcal{T}_h}(\nabla_w\cdot(\bm{Q}_h\bm{u}_t), q_0)_K = \sum\limits_{K\in\mathcal{T}_h}(\nabla\cdot\bm{u}_t, q_0)_K.
\end{equation*}
Using Lemma \ref{Lemma1}, the definition of discrete weak gradient, integration by parts and the definition of $\mathbb{Q}_0$, we get
\begin{align*}
\sum\limits_{K\in\mathcal{T}_h}\kappa(\nabla_wQ_hp, \nabla_wq_h)_K &= \sum\limits_{K\in\mathcal{T}_h}\kappa(\mathbb{Q}_0(\nabla p), \nabla_wq_h)_K\\
&= \sum\limits_{K\in\mathcal{T}_h}\{\kappa\langle \mathbb{Q}_0(\nabla p)\cdot\bm{n}, q_b\rangle_{\partial K}-\kappa(\nabla\cdot\mathbb{Q}_0(\nabla p), q_0)_K\}\\
&= \sum\limits_{K\in\mathcal{T}_h}\{\kappa(\mathbb{Q}_0(\nabla p), \nabla q_0)_K-\kappa\langle \mathbb{Q}_0(\nabla p)\cdot\bm{n}, q_0-q_b\rangle_{\partial K}\}\\
&= \sum\limits_{K\in\mathcal{T}_h}\{\kappa(\nabla p, \nabla q_0)_K-\kappa\langle \mathbb{Q}_0(\nabla p)\cdot\bm{n}, q_0-q_b\rangle_{\partial K}\},
\end{align*}
which implies that
\begin{equation*}
\sum\limits_{K\in\mathcal{T}_h}\kappa(\nabla p, \nabla q_0)_K = \sum\limits_{K\in\mathcal{T}_h}\kappa(\nabla_wQ_hp, \nabla_wq_h)_K+\sum\limits_{K\in\mathcal{T}_h}\kappa\langle \mathbb{Q}_0(\nabla p)\cdot\bm{n}, q_0-q_b\rangle_{\partial K}.
\end{equation*}
Testing \eqref{eq:6} with $q_0$ in $q_h=\{q_0, q_b\}$, and utilizing integration by parts, the boundary conditions and the above estimates, we acquire the second equality of this lemma, which completes the proof.
\end{proof}

\indent Based on Lemma \ref{Lemma2}, we apply the Wheeler's projection method in \cite{article9,article14} to study the optimal order of error estimates. For any $\bm{v}_h\in\bm{V}_h^{\bm{u}}$ and $q_h\in W_h^{p}$, define two elliptic projections $\bm{\widetilde{u}}_h\in\bm{V}_h^{\bm{u}}$ and $\widetilde{p}_h\in W_h^{p}$ such that
\begin{align}
a_{\bm{u}}(\bm{\widetilde{u}}_h,\bm{v}_h)-b(\bm{v}_h,\widetilde{p}_h) = &a_{\bm{u}}(\bm{Q}_h\bm{u},\bm{v}_h)-b(\bm{v}_h,Q_hp)-s_{\bm{u}}(\bm{Q}_h\bm{u},\bm{v}_h)-l_1(\bm{u},\bm{v}_h)-l_2(\bm{u},\bm{v}_h)\notag\\
&+l_3(p,\bm{v}_h),\label{eq:15}\\
a_p(\widetilde{p}_h,q_h) = &a_p(Q_hp,q_h)-s_p(Q_hp,q_h)-l_4(p,q_h).\label{eq:14}
\end{align}

\indent For the numerical analysis of WG, we usually focus on the following error decomposition,
\begin{equation*}
\bm{Q}_h\bm{u}-\bm{u}_h=(\bm{Q}_h\bm{u}-\bm{\widetilde{u}}_h)+(\bm{\widetilde{u}}_h-\bm{u}_h):=\bm{\epsilon}_h+\bm{e}_h,
\end{equation*}
and
\begin{equation*}
Q_hp-p_h=(Q_hp-\widetilde{p}_h)+(\widetilde{p}_h-p_h):=\theta_h+\rho_h.
\end{equation*}

\indent In order to bound the errors $\bm{\epsilon}_h$ and $\theta_h$, we have the following results from \cite{article38, MR3775100}.
\begin{lemma}\cite{article38, MR3775100}\label{Lemma3}
Let $\bm{u}\in\bm{H}^1_{\bm{u}}(\Omega)\cap\bm{H}^{j+2}(\Omega)$ and $p\in H^1_p(\Omega)\cap H^{j+1}(\Omega)$ with any integer $j\geq 1$, for any $\bm{v}_h\in\bm{V}_h$ and $q_h\in W_h$, then
\begin{align*}
|s_{\bm{u}}(\bm{Q}_h\bm{u},\bm{v}_h)| &\leq Ch^{j+1}\|\bm{u}\|_{j+2}|\!|\!| \bm{v}_h|\!|\!|_{\bm{V}},\\
|s_p(Q_hp,q_h)| &\leq Ch^j\|p\|_{j+1}|\!|\!| q_h|\!|\!|_W,\\
|l_1(\bm{u},\bm{v}_h)| &\leq Ch^{j+1}\|\bm{u}\|_{j+2}|\!|\!| \bm{v}_h|\!|\!|_{\bm{V}},\\
|l_2(\bm{u},\bm{v}_h)| &\leq Ch^{j+1}\|\bm{u}\|_{j+2}|\!|\!| \bm{v}_h|\!|\!|_{\bm{V}},\\
|l_3(p,\bm{v}_h)| &\leq Ch^{j+1}\|p\|_{j+1}|\!|\!| \bm{v}_h|\!|\!|_{\bm{V}},\\
|l_4(p,q_h)| &\leq Ch^{j}\|p\|_{j+1}|\!|\!| q_h|\!|\!|_{W}.
\end{align*}
\end{lemma}

\indent The estimates of $\bm{\epsilon}_h$ and $\theta_h$ are provided as follows.
\begin{lemma}\label{Lemma4}
Assume that $\bm{u}\in\bm{H}^1_{\bm{u}}(\Omega)\cap \bm{H}^{j+2}(\Omega)$ and $p\in H_p^1(\Omega)\cap H^{j+1}(\Omega)$ with any integer $j\geq 1$. For any $\bm{v}_h\in\bm{V}_h^{\bm{u}}$ and $q_h\in W_h^p$, there hold
\begin{align*}
|\!|\!| \theta_h|\!|\!|_W &\leq Ch^j\|p\|_{j+1},\\
|\!|\!| \bm{\epsilon}_h|\!|\!|_{\bm{V}} &\leq Ch^{j+1}(\|\bm{u}\|_{j+2}+\|p\|_{j+1}).
\end{align*}
\end{lemma}
\begin{proof}
According to \eqref{eq:14} and Lemma \ref{Lemma3}, we obtain
\begin{align*}
a_p(\theta_h,q_h) &= a_p(Q_hp, q_h)-a_p(\widetilde{p}_h, q_h)\\
&= s_p(Q_hp, q_h)+l_4(p, q_h)\\
&\leq |s_p(Q_hp, q_h)|+|l_4(p, q_h)|\\
&\leq Ch^j\|p\|_{j+1}|\!|\!| q_h|\!|\!|_W.
\end{align*}
Taking $q_h=\theta_h$ in the above inequality, then
\begin{equation*}
|\!|\!| \theta_h|\!|\!|_W \leq Ch^j\|p\|_{j+1}.
\end{equation*}

\indent Because of \eqref{eq:15} and Lemma \ref{Lemma3}, we get
\begin{align*}
a_{\bm{u}}(\bm{\epsilon}_h, \bm{v}_h) &= a_{\bm{u}}(\bm{Q}_h\bm{u}, \bm{v}_h)-a_{\bm{u}}(\bm{\widetilde{u}}_h, \bm{v}_h)\\
&= b(\bm{v}_h, \theta_h)+s_{\bm{u}}(\bm{Q}_h\bm{u}, \bm{v}_h)+l_1(\bm{u},\bm{v}_h)+l_2(\bm{u},\bm{v}_h)-l_3(p,\bm{v}_h)\\
&\leq C|\!|\!| \bm{v}_h|\!|\!|_{\bm{V}}\|\theta_0\|+Ch^{j+1}\|\bm{u}\|_{j+2}|\!|\!|\bm{v}_h|\!|\!|_{\bm{V}}+Ch^{j+1}\|p\|_{j+1}|\!|\!| \bm{v}_h|\!|\!|_{\bm{V}}.
\end{align*}
Since \eqref{eq:14} is a standard WG discretization of Poisson equation, by the duality argument \cite{MR3325251}, we acquire
\begin{equation*}
\|\theta_0\| \leq Ch^{j+1}\|p\|_{j+1}.
\end{equation*}
Hence,
\begin{equation*}
a_{\bm{u}}(\bm{\epsilon}_h, \bm{v}_h) \leq Ch^{j+1}(\|\bm{u}\|_{j+2}+\|p\|_{j+1})|\!|\!| \bm{v}_h|\!|\!|_{\bm{V}}.
\end{equation*}
Taking $\bm{v}_h=\bm{\epsilon}_h$ in the above inequality, then
\begin{equation*}
|\!|\!| \bm{\epsilon}_h|\!|\!|_{\bm{V}} \leq Ch^{j+1}(\|\bm{u}\|_{j+2}+\|p\|_{j+1}),
\end{equation*}
which finishes the proof.
\end{proof}

\indent Next, we derive the error equations of semi-discrete WG method. For $\bm{v}_h\in\bm{V}_h^{\bm{u}}$ and $q_h\in W_h^p$, noticing \eqref{eq:15}, Lemma \ref{Lemma2} and \eqref{eq:3}, it follows that
\begin{align*}
a_{\bm{u}}(\bm{e}_h, \bm{v}_h)-b(\bm{v}_h, \rho_h) &= [a_{\bm{u}}(\bm{\widetilde{u}}_h, \bm{v}_h)-b(\bm{v}_h, \widetilde{p}_h)]-[a_{\bm{u}}(\bm{u}_h, \bm{v}_h)-b(\bm{v}_h, p_h)]\\
&= (\bm{f}, \bm{v}_0)-(\bm{f}, \bm{v}_0)\\
&= 0.
\end{align*}
By virtue of \eqref{eq:4}, \eqref{eq:14} and Lemma \ref{Lemma2}, we get
\begin{align*}
&(c_0\rho_{0,t}, q_0)+b(\bm{e}_{h,t}, q_h)+a_p(\rho_h,q_h)\\
= &[(c_0\widetilde{p}_{0,t}, q_0)+b(\bm{\widetilde{u}}_{h,t}, q_h)+a_p(\widetilde{p}_h,q_h)]-[(c_0p_{0,t}, q_0)+b(\bm{u}_{h,t}, q_h)+a_p(p_h,q_h)]\\
= &(c_0(\widetilde{p}_{0,t}-Q_0p_t), q_0)+b(\bm{\widetilde{u}}_{h,t}-\bm{Q}_h\bm{u}_t, q_h)\\
= &-(c_0\theta_{0,t},q_0)-b(\bm{\epsilon}_{h,t},q_h).
\end{align*}
Thus, we obtain the error equations
\begin{align}
a_{\bm{u}}(\bm{e}_h, \bm{v}_h)-b(\bm{v}_h, \rho_h) &= 0,\label{eq:18}\\
(c_0\rho_{0,t}, q_0)+b(\bm{e}_{h,t}, q_h)+a_p(\rho_h,q_h) &= -(c_0\theta_{0,t},q_0)-b(\bm{\epsilon}_{h,t},q_h).\label{eq:19}
\end{align}

\indent Before the error estimates of semi-discrete WG scheme, we need to supply the following useful lemma.
\begin{lemma}\cite{article32}\label{Lemma0}
Assume that the finite element partition $\mathcal{T}_h$ is shape regular. Then there exists a constant $C$ such that
\begin{equation*}
\|q_0\|^2 \leq C|\!|\!| q_h|\!|\!|^2_W, \qquad \forall\ q_h=\{q_0, q_b\}\in W_h^p.
\end{equation*}
\end{lemma}

\indent Now we are ready for the optimal order convergence estimates as follows.
\begin{theorem}\label{Theorem1}
For any $t\in (0,T]$ and $j \geq 1$, let
\begin{equation*}
(\bm{u}(t); p(t))\in L^{\infty}(0,T;\bm{H}^1_{\bm{u}}(\Omega)\cap\bm{H}^{j+2}(\Omega)) \times L^{\infty}(0,T;H_p^1(\Omega)\cap H^{j+1}(\Omega))
\end{equation*}
be the solution of \eqref{eq:5} and \eqref{eq:6}, and $(\bm{u}_h(t); p_h(t))\in \bm{V}_h^{\bm{u}} \times W_h^p$ be the solution of \eqref{eq:3} and \eqref{eq:4}. Assume that
\begin{equation*}
\bm{u}_t(t)\in L^2(0,T;\bm{H}^{j+2}(\Omega)), \qquad p_t(t)\in L^2(0,T;H^{j+1}(\Omega)),
\end{equation*}
then
\begin{align*}
|\!|\!| \bm{Q}_h\bm{u}(t)-\bm{u}_h(t)|\!|\!|^2_{\bm{V}}+\|Q_0p(t)-p_0(t)\|^2 \leq &Ch^{2j+2}(\|\bm{u}(0)\|^2_{j+2}+\|p(0)\|^2_{j+1}+\|\bm{u}(t)\|^2_{j+2}+\|p(t)\|^2_{j+1}\\
+ &\int_0^t\|\bm{u}_t(\tau)\|^2_{j+2}d\tau+\int_0^t\|p_t(\tau)\|^2_{j+1}d\tau).
\end{align*}
\end{theorem}
\begin{proof}
Choosing $\bm{v}_h=\bm{e}_{h,t}$ and $q_h=\rho_h$ in \eqref{eq:18} and \eqref{eq:19}, respectively, and adding these two equalities yield
\begin{equation*}
a_{\bm{u}}(\bm{e}_h, \bm{e}_{h,t})+(c_0\rho_{0,t}, \rho_0)+|\!|\!| \rho_h|\!|\!|_{W}^2 = -(c_0\theta_{0,t}, \rho_0)-b(\bm{\epsilon}_{h,t}, \rho_h).
\end{equation*}
It follows from Lemma \ref{Lemma0} and the Cauchy-Schwarz inequality that
\begin{align*}
\frac{1}{2}\frac{d}{dt}|\!|\!| \bm{e}_h|\!|\!|_{\bm{V}}^2+\frac{c_0}{2}\frac{d}{dt}\|\rho_0\|^2+|\!|\!| \rho_h|\!|\!|_{W}^2 &= -(c_0\theta_{0,t}, \rho_0)-b(\bm{\epsilon}_{h,t},\rho_h)\\
&\leq c_0\|\theta_{0,t}\|\|\rho_0\|+C|\!|\!| \bm{\epsilon}_{h,t}|\!|\!|_{\bm{V}}\|\rho_0\|\\
&\leq c_0\|\theta_{0,t}\||\!|\!|\rho_h|\!|\!|_W+C|\!|\!| \bm{\epsilon}_{h,t}|\!|\!|_{\bm{V}}|\!|\!|\rho_h|\!|\!|_W\\
&\leq C\|\theta_{0,t}\|^2+C|\!|\!|\bm{\epsilon}_{h,t}|\!|\!|_{\bm{V}}^2+C_1|\!|\!|\rho_h|\!|\!|_W^2.
\end{align*}
Supposing $0<C_1\leq 1$ and integrating the both sides of inequality with respect to $t$, together with Lemma \ref{Lemma4}, we write
\begin{align*}
|\!|\!|\bm{e}_h(t)|\!|\!|_{\bm{V}}^2+\|\rho_0(t)\|^2 &\leq |\!|\!|\bm{e}_h(0)|\!|\!|_{\bm{V}}^2+\|\rho_0(0)\|^2+C\int_0^t\|\theta_{0,t}(\tau)\|^2d\tau+C\int_0^t|\!|\!|\bm{\epsilon}_{h,t}(\tau)|\!|\!|_{\bm{V}}^2d\tau\\
&\leq |\!|\!|\bm{e}_h(0)|\!|\!|_{\bm{V}}^2+\|\rho_0(0)\|^2 + Ch^{2j+2}(\int_0^t\|\bm{u}_{t}(\tau)\|_{j+2}^2d\tau+\int_0^t\|p_{t}(\tau)\|_{j+1}^2d\tau).
\end{align*}
Because of the error estimate of $L^2$ projection operator and Lemma \ref{Lemma4}, we arrive at
\begin{equation}\label{eq:30}
\begin{split}
|\!|\!|\bm{e}_h(0)|\!|\!|^2_{\bm{V}} &= |\!|\!|(\bm{Q}_h\bm{u}(0)-\bm{u}_h(0))-\bm{\epsilon}_h(0)|\!|\!|^2_{\bm{V}}\\
&\leq C|\!|\!|\bm{u}(0)-\bm{u}_h(0)|\!|\!|^2_{\bm{V}}+C|\!|\!|\bm{u}(0)-\bm{Q}_h\bm{u}(0)|\!|\!|^2_{\bm{V}}+C|\!|\!|\bm{\epsilon}_h(0)|\!|\!|^2_{\bm{V}}\\
&\leq Ch^{2j+2}(\|\bm{u}(0)\|_{j+2}^2+\|p(0)\|_{j+1}^2).
\end{split}
\end{equation}
Similarly, there holds
\begin{equation}\label{eq:31}
\begin{split}
\|\rho_0(0)\|^2 &= \|(Q_0p(0)-p_0(0))-\theta_0(0)\|^2\\
&\leq C\|p(0)-p_0(0)\|^2+C\|p(0)-Q_0p(0)\|^2+C\|\theta_0(0)\|^2\\
&\leq Ch^{2j+2}\|p(0)\|^2_{j+1}.
\end{split}
\end{equation}
Therefore,
\begin{align*}
&|\!|\!| \bm{Q}_h\bm{u}(t)-\bm{u}_h(t)|\!|\!|^2_{\bm{V}}+\|Q_0p(t)-p_0(t)\|^2\\
\leq &C(|\!|\!| \bm{\epsilon}_h(t)|\!|\!|^2_{\bm{V}}+|\!|\!| \bm{e}_h(t)|\!|\!|^2_{\bm{V}}+\|\theta_0(t)\|^2+\|\rho_0(t)\|^2)\\
\leq &Ch^{2j+2}(\|\bm{u}(0)\|^2_{j+2}+\|p(0)\|^2_{j+1}+\|\bm{u}(t)\|^2_{j+2}+\|p(t)\|^2_{j+1}+\int_0^t\|\bm{u}_t(\tau)\|^2_{j+2}d\tau+\int_0^t\|p_t(\tau)\|^2_{j+1}d\tau),
\end{align*}
where Lemma \ref{Lemma4} is applied, and the proof is completed.
\end{proof}

\begin{theorem}
Under the assumption of Theorem \ref{Theorem1} with $c_0>0$, we have
\begin{equation*}
|\!|\!| Q_hp(t)-p_h(t)|\!|\!|^2_W \leq Ch^{2j}(\|p(0)\|_{j+1}^2+\|p(t)\|_{j+1}^2)+Ch^{2j+2}(\int_0^t\|p_t(\tau)\|_{j+1}^2d\tau+\int_0^t\|\bm{u}_t(\tau)\|_{j+2}^2d\tau).
\end{equation*}
\end{theorem}
\begin{proof}
First, we differentiate \eqref{eq:18} with respect to $t$,
\begin{equation}\label{eq:20}
a_{\bm{u}}(\bm{e}_{h,t}, \bm{v}_h)-b(\bm{v}_h, \rho_{h,t})=0.
\end{equation}
Taking $\bm{v}_h=\bm{e}_{h,t}$ and $q_h=\rho_{h,t}$ in \eqref{eq:20} and \eqref{eq:19}, respectively, and adding,
\begin{align*}
|\!|\!|\bm{e}_{h,t}|\!|\!|_{\bm{V}}^2+c_0\|\rho_{0,t}\|^2+\frac{1}{2}\frac{d}{dt}|\!|\!|\rho_h|\!|\!|_W^2 &= -(c_0\theta_{0,t},\rho_{0,t})-b(\bm{\epsilon}_{h,t}, \rho_{h,t})\\
&\leq c_0\|\theta_{0,t}\|\|\rho_{0,t}\|+C|\!|\!|\bm{\epsilon}_{h,t}|\!|\!|_{\bm{V}}\|\rho_{0,t}\|\\
&\leq C\|\theta_{0,t}\|^2+C|\!|\!|\bm{\epsilon}_{h,t}|\!|\!|^2_{\bm{V}}+C_2\|\rho_{0,t}\|^2.
\end{align*}
Assuming $0<C_2\leq c_0$ and integrating with respect to $t$, it follows from Lemma \ref{Lemma4} that
\begin{align*}
|\!|\!|\rho_{h}(t)|\!|\!|_{W}^2 &\leq |\!|\!|\rho_h(0)|\!|\!|_W^2+C\int_0^t\|\theta_{0,t}(\tau)\|^2d\tau+C\int_0^t|\!|\!|\bm{\epsilon}_{h,t}(\tau)|\!|\!|_{\bm{V}}^2d\tau\\
&\leq |\!|\!|\rho_{h}(0)|\!|\!|_W^2+Ch^{2j+2}(\int_0^t\|p_t(\tau)\|_{j+1}^2d\tau+\int^t_0\|\bm{u}_t(\tau)\|_{j+2}^2d\tau).
\end{align*}
Using the error estimate of $L^2$ projection operator and Lemma \ref{Lemma4}, we provide
\begin{equation}\label{eq:33}
\begin{split}
|\!|\!| \rho_h(0)|\!|\!|_W^2 &\leq C(|\!|\!| p(0)-p_h(0)|\!|\!|_W^2+|\!|\!| p(0)-Q_hp(0)|\!|\!|_W^2+|\!|\!| \theta_h(0)|\!|\!|_W^2)\\
&\leq Ch^{2j}\|p(0)\|_{j+1}^2.
\end{split}
\end{equation}
Hence,
\begin{align*}
|\!|\!| Q_hp(t)-p_h(t)|\!|\!|^2_W &\leq C(|\!|\!| \theta_h(t)|\!|\!|^2_W+|\!|\!| \rho_h(t)|\!|\!|^2_W)\\
&\leq Ch^{2j}\|p(t)\|_{j+1}^2+C|\!|\!|\rho_h(0)|\!|\!|_W^2+Ch^{2j+2}(\int_0^t\|p_t(\tau)\|_{j+1}^2d\tau+\int_0^t\|\bm{u}_t(\tau)\|_{j+2}^2d\tau)\\
&\leq Ch^{2j}(\|p(0)\|_{j+1}^2+\|p(t)\|_{j+1}^2)+Ch^{2j+2}(\int_0^t\|p_t(\tau)\|_{j+1}^2d\tau+\int_0^t\|\bm{u}_t(\tau)\|_{j+2}^2d\tau),
\end{align*}
where Lemma \ref{Lemma4} is utilized, and we finish the proof.
\end{proof}

\subsection{Discrete time WG method}

\indent In this section, we estimate the errors of fully discrete WG method. Similarly to the semi-discrete problem, we separate $\bm{Q}_h\bm{u}(t_n)-\bm{u}_h^n$ and $Q_hp(t_n)-p_h^n$ into two parts, respectively,
\begin{align*}
\bm{Q}_h\bm{u}(t_n)-\bm{u}_h^n = &(\bm{Q}_h\bm{u}(t_n)-\bm{\widetilde{u}}_h(t_n))+(\bm{\widetilde{u}}_h(t_n)-\bm{u}_h^n)\\
:= &\bm{\epsilon}_h(t_n)+\bm{e}_h^n,
\end{align*}
and
\begin{align*}
Q_hp(t_n)-p_h^n = &(Q_hp(t_n)-\widetilde{p}_h(t_n))+(\widetilde{p}_h(t_n)-p_h^n)\\
:= &\theta_h(t_n)+\rho_h^n.
\end{align*}
Then we obtain the discrete time error equations,
\begin{align*}
a_{\bm{u}}(\bm{e}_h^n, \bm{v}_h)-b(\bm{v}_h, \rho_h^n) &= [a_{\bm{u}}(\bm{\widetilde{u}}_h(t_n), \bm{v}_h)-b(\bm{v}_h, \widetilde{p}_h(t_n))]-[a_{\bm{u}}(\bm{u}_h^n, \bm{v}_h)-b(\bm{v}_h, p_h^n)]\\
&= (\bm{f}(t_n), \bm{v}_0)-(\bm{f}(t_n), \bm{v}_0)\\
&= 0,
\end{align*}
and
\begin{align*}
&(c_0\partial_\tau \rho_0^n, q_0)+b(\partial_\tau \bm{e}_h^n, q_h)+a_p(\rho_h^n, q_h)\\
= &[(c_0\partial_\tau \widetilde{p}_0(t_n), q_0)+b(\partial_\tau \bm{\widetilde{u}}_h(t_n), q_h)+a_p(\widetilde{p}_h(t_n), q_h)]-[(c_0\partial_\tau p_0^n, q_0)+b(\partial_\tau \bm{u}_h^n, q_h)+a_p(p_h^n, q_h)]\\
= &[(c_0\partial_\tau \widetilde{p}_0(t_n), q_0)+b(\partial_\tau \bm{\widetilde{u}}_h(t_n), q_h)+a_p(\widetilde{p}_h(t_n), q_h)]-[(g(t_n), q_0)+\langle \gamma(t_n), q_b\rangle_{\Gamma_{p, N}}]\\
= &c_0(\partial_\tau \widetilde{p}_0(t_n)-Q_0p_t(t_n), q_0)+b(\partial_\tau \bm{\widetilde{u}}_h(t_n)-\bm{Q}_h\bm{u}_t(t_n), q_h)\\
= &[-c_0(\partial_\tau \theta_0(t_n), q_0)-b(\partial_\tau \bm{\epsilon}_h(t_n), q_h)]+[c_0(\partial_\tau Q_0p(t_n)-Q_0p_t(t_n), q_0)+b(\partial_\tau \bm{Q}_h\bm{u}(t_n)-\bm{Q}_h\bm{u}_t(t_n), q_h)].
\end{align*}
For convenience, we set $J^n_{p0}:=\partial_\tau Q_0p(t_n)-Q_0p_t(t_n)$ and $\bm{J}_{\bm{u}}^n:=\partial_\tau \bm{Q}_h\bm{u}(t_n)-\bm{Q}_h\bm{u}_t(t_n)$, then the error equations for the fully discrete problem are given as follows,
\begin{align}
a_{\bm{u}}(\bm{e}_h^n, \bm{v}_h)-b(\bm{v}_h, \rho_h^n) = &0,\label{eq:23}\\
(c_0\partial_\tau \rho_0^n, q_0)+b(\partial_\tau \bm{e}_h^n, q_h)+a_p(\rho_h^n, q_h) = &[-c_0(\partial_\tau \theta_0(t_n), q_0)-b(\partial_\tau \bm{\epsilon}_h(t_n), q_h)]\notag\\
&+[c_0(J^n_{p0}, q_0)+b(\bm{J}_{\bm{u}}^n, q_h)].\label{eq:24}
\end{align}

\indent The optimal order error estimates for the fully discrete scheme are given in the next two theorems.
\begin{theorem}\label{fullydiscrete}
For $j\geq 1$, let
\begin{equation*}
(\bm{u}(t); p(t))\in L^{\infty}(0,T;\bm{H}_{\bm{u}}^1(\Omega)\cap\bm{H}^{j+2}(\Omega)) \times L^{\infty}(0,T;H_p^1(\Omega)\cap H^{j+1}(\Omega))
\end{equation*}
be the solution of \eqref{eq:5} and \eqref{eq:6}, and $(\bm{u}_h^n; p_h^n)\in \bm{V}_h^{\bm{u}} \times W_h^p$ be the solution of \eqref{eq:21} and \eqref{eq:22}. Suppose that
\begin{align*}
\bm{u}_t(t) &\in L^2(0,T;\bm{H}^{j+2}(\Omega)), \quad \bm{u}_{tt}(t)\in L^2(0,T;\bm{H}_{\bm{u}}^1(\Omega)),\\
p_t(t) &\in L^2(0,T;H^{j+1}(\Omega)),\ \quad p_{tt}(t)\in L^2(0,T; H_p^1(\Omega)).
\end{align*}
Then there holds
\begin{align*}
|\!|\!|\bm{Q}_h\bm{u}(t_n)-\bm{u}_h^n|\!|\!|_{\bm{V}}^2+\|Q_0p(t_n)-p_0^n\|^2 \leq &C\tau^2\int_0^{t_n}(\|\bm{u}_{tt}\|_1^2+\|p_{tt}\|^2)d\tau+Ch^{2j+2}(\|\bm{u}(0)\|_{j+2}^2+\|p(0)\|_{j+1}^2\\
&+\|\bm{u}(t_n)\|_{j+2}^2+\|p(t_n)\|_{j+1}^2+\int_0^{t_n}(\|\bm{u}_t\|_{j+2}^2+\|p_t\|_{j+1}^2)d\tau).
\end{align*}
\end{theorem}
\begin{proof}
Choosing $\bm{v}_h=\partial_\tau \bm{e}_h^n$ and $q_h=\rho_h^n$ in \eqref{eq:23} and \eqref{eq:24} and counting up,
\begin{equation*}
a_{\bm{u}}(\bm{e}_h^n, \partial_\tau \bm{e}_h^n)+c_0(\partial_\tau \rho_0^n, \rho_0^n)+|\!|\!|\rho_h^n|\!|\!|_W^2=[-c_0(\partial_\tau \theta_0(t_n), \rho_0^n)-b(\partial_\tau \bm{\epsilon}_h(t_n), \rho_h^n)]+[c_0(J_{p0}^n, \rho_0^n)+b(\bm{J}_{\bm{u}}^n, \rho_h^n)].
\end{equation*}
Since
\begin{equation*}
a_{\bm{u}}(\bm{e}_h^n, \partial_\tau \bm{e}_h^n)=\frac{1}{2}\partial_\tau a_{\bm{u}}(\bm{e}_h^n, \bm{e}_h^n)+\frac{\tau}{2}a_{\bm{u}}(\partial_\tau \bm{e}_h^n, \partial_\tau \bm{e}_h^n),
\end{equation*}
and
\begin{equation*}
(\partial_\tau \rho_0^n, \rho_0^n)=\frac{1}{2}\partial_\tau (\rho_0^n, \rho_0^n)+\frac{\tau}{2}(\partial_\tau \rho_0^n, \partial_\tau \rho_0^n),
\end{equation*}
together with the Cauchy-Schwarz inequality and Lemma \ref{Lemma0}, then
\begin{align*}
&\frac{1}{2}\partial_\tau |\!|\!|\bm{e}_h^n|\!|\!|^2_{\bm{V}}+\frac{\tau}{2}|\!|\!|\partial_\tau \bm{e}_h^n|\!|\!|^2_{\bm{V}}+\frac{c_0}{2}\partial_\tau \|\rho_0^n\|^2+\frac{c_0\tau}{2}\|\partial_\tau \rho_0^n\|^2+|\!|\!|\rho_h^n|\!|\!|_W^2\\
\leq &C(\|\partial_\tau \theta_0(t_n)\|\|\rho_0^n\|+|\!|\!|\partial_\tau \bm{\epsilon}_h(t_n)|\!|\!|_{\bm{V}}\|\rho_0^n\|+\|J_{p0}^n\|\|\rho_0^n\|+|\!|\!|\bm{J}_{\bm{u}}^n|\!|\!|_{\bm{V}}\|\rho_0^n\|)\\
\leq &C(\|\partial_\tau \theta_0(t_n)\|^2+|\!|\!|\partial_\tau \bm{\epsilon}_h(t_n)|\!|\!|_{\bm{V}}^2+\|J_{p0}^n\|^2+|\!|\!|\bm{J}_{\bm{u}}^n|\!|\!|_{\bm{V}}^2)+C_3|\!|\!|\rho_h^n|\!|\!|_W^2.
\end{align*}
Let $0<C_3\leq 1$, and we find
\begin{equation*}
\frac{1}{2}\partial_\tau |\!|\!|\bm{e}_h^n|\!|\!|_{\bm{V}}^2+\frac{c_0}{2}\partial_\tau \|\rho_0^n\|^2 \leq C(\|\partial_\tau \theta_0(t_n)\|^2+|\!|\!|\partial_\tau \bm{\epsilon}_h(t_n)|\!|\!|^2_{\bm{V}}+\|J_{p0}^n\|^2+|\!|\!|\bm{J}_{\bm{u}}^n|\!|\!|_{\bm{V}}^2),
\end{equation*}
i.e.,
\begin{equation*}
|\!|\!|\bm{e}_h^n|\!|\!|_{\bm{V}}^2+\|\rho_0^n\|^2 \leq |\!|\!|\bm{e}_h^{n-1}|\!|\!|^2_{\bm{V}}+\|\rho_0^{n-1}\|^2+C\tau(\|\partial_\tau \theta_0(t_n)\|^2+|\!|\!|\partial_\tau \bm{\epsilon}_h(t_n)|\!|\!|^2_{\bm{V}}+\|J_{p0}^n\|^2+|\!|\!|\bm{J}_{\bm{u}}^n|\!|\!|^2_{\bm{V}}).
\end{equation*}
It follows by induction that
\begin{equation}\label{eq:25}
|\!|\!|\bm{e}_h^n|\!|\!|_{\bm{V}}^2+\|\rho_0^n\|^2 \leq |\!|\!|\bm{e}_h^{0}|\!|\!|^2_{\bm{V}}+\|\rho_0^{0}\|^2+\sum\limits_{i=1}^{n}C\tau(\|\partial_\tau \theta_0(t_i)\|^2+|\!|\!|\partial_\tau \bm{\epsilon}_h(t_i)|\!|\!|^2_{\bm{V}}+\|J_{p0}^i\|^2+|\!|\!|\bm{J}_{\bm{u}}^i|\!|\!|^2_{\bm{V}}).
\end{equation}
Next, we estimate the four terms $\sum\limits_{i=1}^{n}C\tau\|\partial_\tau \theta_0(t_i)\|^2$, $\sum\limits_{i=1}^{n}C\tau|\!|\!|\partial_\tau \bm{\epsilon}_h(t_i)|\!|\!|^2_{\bm{V}}$, $\sum\limits_{i=1}^{n}C\tau\|J_{p0}^i\|^2$ and $\sum\limits_{i=1}^{n}C\tau|\!|\!|\bm{J}_{\bm{u}}^i|\!|\!|^2_{\bm{V}}$, respectively. Firstly, let us focus on $\sum\limits_{i=1}^{n}C\tau\|\partial_\tau \theta_0(t_i)\|^2$. Since
\begin{equation*}
\partial_\tau \theta_0(t_i)=\frac{1}{\tau}\int_{t_{i-1}}^{t_i}\theta_{0,t}d\tau,
\end{equation*}
combined with Lemma \ref{Lemma4}, we present
\begin{equation}\label{eq:26}
\sum\limits_{i=1}^nC\tau\|\partial_\tau \theta_0(t_i)\|^2 \leq \sum\limits_{i=1}^nC\int_{t_{i-1}}^{t_i}\|\theta_{0,t}\|^2d\tau \leq Ch^{2j+2}\int_0^{t_n}\|p_t\|_{j+1}^2d\tau.
\end{equation}
Similarly, we have
\begin{equation}\label{eq:27}
\sum\limits_{i=1}^nC\tau|\!|\!|\partial_\tau \bm{\epsilon}_h(t_i)|\!|\!|^2_{\bm{V}} \leq \sum\limits_{i=1}^nC\int_{t_{i-1}}^{t_i}|\!|\!|\bm{\epsilon}_{h,t}|\!|\!|^2_{\bm{V}}d\tau \leq Ch^{2j+2}\int_0^{t_n}(\|\bm{u}_t\|_{j+2}^2+\|p_t\|_{j+1}^2)d\tau.
\end{equation}
In order to bound $\sum\limits_{i=1}^nC\tau\|J_{p0}^i\|^2$, we use
\begin{equation*}
J_{p0}^i = \partial_\tau Q_0p(t_i)-Q_0p_t(t_i) = \frac{1}{\tau}Q_0(p(t_i)-p(t_{i-1})-\tau p_t(t_i)) = \frac{1}{\tau}Q_0(\int_{t_{i-1}}^{t_i}(-\tau+t_{i-1})p_{tt}(\tau)d\tau),
\end{equation*}
therefore,
\begin{equation}\label{eq:28}
\sum\limits_{i=1}^nC\tau\|J_{p0}^i\|^2 \leq \sum\limits_{i=1}^nC\int_{t_{i-1}}^{t_i}\|(-\tau+t_{i-1})p_{tt}(\tau)\|^2d\tau \leq C\tau^2\int_0^{t_n}\|p_{tt}\|^2d\tau.
\end{equation}
Likewise, we obtain
\begin{equation}\label{eq:29}
\sum\limits_{i=1}^nC\tau|\!|\!|\bm{J}_{\bm{u}}^i|\!|\!|^2_{\bm{V}} \leq \sum\limits_{i=1}^nC\int_{t_{i-1}}^{t_i}\|(-\tau+t_{i-1})\bm{u}_{tt}(\tau)\|^2_1d\tau \leq C\tau^2\int_0^{t_n}\|\bm{u}_{tt}\|_1^2d\tau.
\end{equation}
Substituting \eqref{eq:26}, \eqref{eq:27}, \eqref{eq:28} and \eqref{eq:29} into \eqref{eq:25}, we get
\begin{align*}
|\!|\!|\bm{e}_h^n|\!|\!|^2_{\bm{V}}+\|\rho_0^n\|^2 \leq &|\!|\!|\bm{e}_h^0|\!|\!|^2_{\bm{V}}+\|\rho_0^0\|^2+C\tau^2\int_0^{t_n}(\|\bm{u}_{tt}\|_1^2+\|p_{tt}\|^2)d\tau\\
&+Ch^{2j+2}\int_0^{t_n}(\|\bm{u}_t\|_{j+2}^2+\|p_t\|_{j+1}^2)d\tau.
\end{align*}
Making use of Lemma \ref{Lemma4}, \eqref{eq:30} and \eqref{eq:31}, we arrive at
\begin{align*}
&|\!|\!| \bm{Q}_h\bm{u}(t_n)-\bm{u}_h^n|\!|\!|^2_{\bm{V}}+\|Q_0p(t_n)-p_0^n\|^2\\
\leq &C(|\!|\!|\bm{\epsilon}_h(t_n)|\!|\!|^2_{\bm{V}}+|\!|\!|\bm{e}_h^n|\!|\!|^2_{\bm{V}}+\|\theta_0(t_n)\|^2+\|\rho_0^n\|^2)\\
\leq &Ch^{2j+2}(\|\bm{u}(t_n)\|_{j+2}^2+\|p(t_n)\|_{j+1}^2)+C(|\!|\!|\bm{e}_h^0|\!|\!|_{\bm{V}}^2+\|\rho_0^0\|^2)+C\tau^2\int_0^{t_n}(\|\bm{u}_{tt}\|_1^2+\|p_{tt}\|^2)d\tau\\
&+Ch^{2j+2}\int_0^{t_n}(\|\bm{u}_t\|_{j+2}^2+\|p_t\|_{j+1}^2)d\tau\\
\leq &C\tau^2\int_0^{t_n}(\|\bm{u}_{tt}\|_1^2+\|p_{tt}\|^2)d\tau+Ch^{2j+2}(\|\bm{u}(0)\|_{j+2}^2+\|p(0)\|_{j+1}^2+\|\bm{u}(t_n)\|_{j+2}^2+\|p(t_n)\|_{j+1}^2\\
&+\int_0^{t_n}(\|\bm{u}_t\|_{j+2}^2+\|p_t\|_{j+1}^2)d\tau),
\end{align*}
which completes the proof.
\end{proof}

\begin{theorem}\label{theorem4.9}
Under the assumption of Theorem \ref{fullydiscrete} together with $c_0>0$, we have the following estimate
\begin{align*}
|\!|\!| Q_hp(t_n)-p_h^n|\!|\!|_W^2 \leq &Ch^{2j}(\|p(0)\|_{j+1}^2+\|p(t_n)\|_{j+1}^2)+Ch^{2j+2}\int_0^{t_n}(\|\bm{u}_t\|_{j+2}^2+\|p_t\|_{j+1}^2)d\tau\\
&+C\tau^2\int_0^{t_n}(\|\bm{u}_{tt}\|_1^2+\|p_{tt}\|^2)d\tau.
\end{align*}
\end{theorem}
\begin{proof}
Applying the backward Euler method to approximate the time derivative in \eqref{eq:20},
\begin{equation}\label{eq:32}
a_{\bm{u}}(\partial_\tau \bm{e}_h^n, \bm{v}_h)-b(\bm{v}_h, \partial_\tau \rho_h^n) = 0.
\end{equation}
Taking $\bm{v}_h=\partial_\tau \bm{e}_h^n$ and $q_h=\partial_\tau \rho_h^n$ in \eqref{eq:32} and \eqref{eq:24}, and adding,
\begin{align*}
|\!|\!| \partial_\tau \bm{e}_h^n|\!|\!|^2_{\bm{V}}+c_0\|\partial_\tau \rho_0^n\|^2+a_p(\rho_h^n, \partial_\tau \rho_h^n) = &[-c_0(\partial_\tau \theta_0(t_n), \partial_\tau \rho_0^n)-b(\partial_\tau \bm{\epsilon}_h(t_n), \partial_\tau \rho_h^n)]\\
&+[c_0(J_{p0}^n, \partial_\tau \rho_0^n)+b(\bm{J}_{\bm{u}}^n, \partial_\tau \rho_h^n)].
\end{align*}
Since
\begin{equation*}
a_p(\rho_h^n, \partial_\tau \rho_h^n)=\frac{1}{2}\partial_\tau a_p(\rho_h^n, \rho_h^n)+\frac{\tau}{2}a_p(\partial_\tau \rho_h^n, \partial_\tau \rho_h^n),
\end{equation*}
it follows from the Cauchy-Schwarz inequality that
\begin{align*}
&|\!|\!| \partial_\tau \bm{e}_h^n|\!|\!|^2_{\bm{V}}+c_0\|\partial_\tau \rho_0^n\|^2+\frac{1}{2}\partial_\tau |\!|\!| \rho_h^n|\!|\!|_W^2+\frac{\tau}{2}|\!|\!| \partial_\tau \rho_h^n|\!|\!|_W^2\\
= &[-c_0(\partial_\tau \theta_0(t_n), \partial_\tau \rho_0^n)-b(\partial_\tau \bm{\epsilon}_h(t_n), \partial_\tau \rho_h^n)]+[c_0(J_{p0}^n, \partial_\tau \rho_0^n)+b(\bm{J}_{\bm{u}}^n, \partial_\tau \rho_h^n)]\\
\leq &C(\|\partial_\tau \theta_0(t_n)\|^2+|\!|\!|\partial_\tau \bm{\epsilon}_h(t_n)|\!|\!|_{\bm{V}}^2+\|J_{p0}^n\|^2+|\!|\!|\bm{J}_{\bm{u}}^n|\!|\!|^2_{\bm{V}})+C_4\|\partial_\tau \rho_0^n\|^2.
\end{align*}
Let $0<C_4\leq c_0$, then
\begin{equation*}
|\!|\!| \rho_h^n|\!|\!|_W^2 \leq |\!|\!| \rho_h^{n-1}|\!|\!|_W^2+C\tau(\|\partial_\tau \theta_0(t_n)\|^2+|\!|\!| \partial_\tau \bm{\epsilon}_h(t_n)|\!|\!|^2_{\bm{V}}+\|J_{p0}^n\|^2+|\!|\!|\bm{J}_{\bm{u}}^n|\!|\!|_{\bm{V}}^2).
\end{equation*}
Utilizing the iteration method, \eqref{eq:33}, \eqref{eq:26}, \eqref{eq:27}, \eqref{eq:28} and \eqref{eq:29}, we obtain
\begin{align*}
|\!|\!|\rho_h^n|\!|\!|_W^2 &\leq |\!|\!|\rho_h^0|\!|\!|_W^2+\sum\limits_{i=1}^nC\tau(\|\partial_\tau \theta_0(t_i)\|^2+|\!|\!|\partial_\tau \bm{\epsilon}_h(t_i)|\!|\!|_{\bm{V}}^2+\|J_{p0}^i\|^2+|\!|\!|\bm{J}_{\bm{u}}^i|\!|\!|^2_{\bm{V}})\\
&\leq Ch^{2j}\|p(0)\|_{j+1}^2+Ch^{2j+2}\int_0^{t_n}(\|\bm{u}_t\|_{j+2}^2+\|p_t\|_{j+1}^2)d\tau+C\tau^2\int_0^{t_n}(\|\bm{u}_{tt}\|_1^2+\|p_{tt}\|^2)d\tau,
\end{align*}
which gives, combined with Lemma \ref{Lemma4},
\begin{align*}
|\!|\!| Q_hp(t_n)-p_h^n|\!|\!|_W^2 \leq &C(|\!|\!| \theta_h(t_n)|\!|\!|_W^2+|\!|\!| \rho_h^n|\!|\!|^2_W)\\
\leq &Ch^{2j}(\|p(0)\|_{j+1}^2+\|p(t_n)\|_{j+1}^2)+Ch^{2j+2}\int_0^{t_n}(\|\bm{u}_t\|_{j+2}^2+\|p_t\|_{j+1}^2)d\tau\\
&+C\tau^2\int_0^{t_n}(\|\bm{u}_{tt}\|_1^2+\|p_{tt}\|^2)d\tau.
\end{align*}
The proof is finished.
\end{proof}

\section{Numerical experiments}\label{NumericalExperiments}

\indent In this section, we carry out some numerical examples from two aspects: (1) Our proposed methods are flexible in the selections of mesh; (2) The locking problem is overcome by the presented WG methods. Throughout this section, we consider the system \eqref{eq:5} and \eqref{eq:6} on a two-dimensional domain $\Omega=(0,1)^2$, with the Dirichlet boundary conditions \eqref{DirichletU} and \eqref{DirichletP} for $\bm{u}$ and $p$ on the entire boundary, respectively. The parameters are taken as $c_0=1$, $\kappa=1$, $\mu=1$ and the final time $T=1$. For $\lambda$, we separately test two cases that $\lambda=1$ and $\lambda=1, 10^{4}, 10^{8}$ in the following two subsections. For the weak finite element spaces, we choose $j=1$. Specifically, we adopt the following discrete spaces,
\begin{align*}
\bm{V}_h &:= \{\bm{v}_{h}=\{\bm{v}_0, \bm{v}_b\}: \{\bm{v}_0, \bm{v}_b\}|_K\in[P_{2}(K)]^2 \times [P_1(e)]^2, \ K\in\mathcal{T}_h,\ e\subset \partial K\},\\
W_h &:= \{q_{h}=\{q_0, q_b\}: \{q_0, q_b\}|_K\in P_{1}(K) \times P_{0}(e), \ K\in\mathcal{T}_h,\ e\subset \partial K\},
\end{align*}
and the weak differential operators are computed by
\begin{align*}
(\nabla_{w,K}\cdot\bm{v}_h, \psi)_K &= -(\bm{v}_0, \nabla\psi)_K+\langle\bm{v}_b\cdot\bm{n}, \psi\rangle_{\partial K}, \quad \forall\ \psi\in P_1(K),\\
(\nabla_{w,K}\bm{v}_h, \phi)_K &= -(\bm{v}_0, \nabla\cdot\phi)_K+\langle\bm{v}_b, \phi\cdot\bm{n}\rangle_{\partial K}, \quad \forall\ \phi\in[P_1(K)]^{2\times 2},\\
(\nabla_{w,K}q_h, \bm{\zeta})_K &= -(q_0, \nabla\cdot\bm{\zeta})_K+\langle q_b, \bm{\zeta}\cdot\bm{n}\rangle_{\partial K}, \quad \forall\ \bm{\zeta}\in[P_{0}(K)]^2.
\end{align*}
As in the previous section, we use $\bm{u}_h^n=\{\bm{u}_0^n, \bm{u}_b^n\}\in \bm{V}_h^{\bm{u}}$ and $p_h^n=\{p_0^n, p_b^n\}\in W_h^p$, given by \eqref{eq:21} and \eqref{eq:22}, to denote the approximate solution of $\bm{u}(t_n)$ and $p(t_n)$, respectively. The $L^2$-norm for $\bm{Q}_0\bm{u}(t_n)-\bm{u}_0^n$, the $|\!|\!|\cdot|\!|\!|_{\bm{V}}$-norm for $\bm{Q}_h\bm{u}(t_n)-\bm{u}_h^n$, the $L^2$-norm for $Q_0p(t_n)-p_0^n$ and the $|\!|\!|\cdot|\!|\!|_W$-norm for $Q_hp(t_n)-p_h^n$ are utilized to illustrate the numerical results.\\

\subsection{Tests for convergence orders on different meshes}\label{cotodm}

\indent In this subsection, we accomplish the numerical computations and estimate the convergence orders on triangular meshes, rectangular meshes and hybrid polygonal meshes, respectively. The Lam\'e constant $\lambda=1$ is chosen. The right-hand side terms $\bm{f}$ and $g$ of \eqref{eq:5} and \eqref{eq:6} are selected according to the analytical solution which is given as follows,
\begin{align*}
\displaystyle\bm{u} &= \binom{10x^2(1-x)^2y(1-y)(1-2y)\exp(-t)}{-10x(1-x)(1-2x)y^2(1-y)^2\exp(-2t)},\\
p &= 10x^2(1-x)^2y(1-y)(1-2y)\exp(-3t).
\end{align*}

\subsubsection{Triangular meshes}

\indent A uniform triangular mesh is considered on the two-dimensional domain $\Omega=(0,1)^2$, and we test the convergence orders with $h=\frac{\sqrt{2}}{2}, \frac{\sqrt{2}}{4}, \frac{\sqrt{2}}{8}, \frac{\sqrt{2}}{16}, \frac{\sqrt{2}}{32}, \frac{\sqrt{2}}{64}$ and the time step $\tau=h^2$. Table \ref{table1} and \ref{table2} describe that the error convergence orders of the $L^2$-norm and the $|\!|\!|\cdot|\!|\!|_{\bm{V}}$-norm for $\bm{u}$ are severally $O(h^3)$ and $O(h^2)$, and the ones of the $L^2$-norm and the $|\!|\!|\cdot|\!|\!|_{W}$-norm for $p$ are $O(h^2)$ and $O(h)$, respectively. These optimal convergence orders verify our theoretical results in Theorem \ref{fullydiscrete} and \ref{theorem4.9}.
\begin{table}[!htbp]
\caption{WG error convergence orders for $\bm{u}$ with $\lambda=1$ and $\tau=h^2$ on uniform triangular meshes}\label{table1}
  \begin{center}
  \begin{tabular}{ccccc}
   \hline
 $h$&    $\|\bm{Q}_0\bm{u}(t_n)-\bm{u}_0^n\|$     &Order         & $|\!|\!|\bm{Q}_h\bm{u}(t_n)-\bm{u}_h^n|\!|\!|_{\bm{V}}$   & Order \\ \hline
 $\frac{\sqrt{2}}{2}$&              9.1014E-03&                    -      &              7.0190E-02                                           & -        \\
 $\frac{\sqrt{2}}{4}$&              1.3362E-03&                    2.7679 &              2.1489E-02                                           & 1.7077    \\
 $\frac{\sqrt{2}}{8}$&              1.7620E-04&                    2.9229 &              5.7772E-03                                           & 1.8952    \\
 $\frac{\sqrt{2}}{16}$&             2.2603E-05&                    2.9626 &              1.4845E-03                                           & 1.9604      \\
 $\frac{\sqrt{2}}{32}$&             2.9438E-06&                    2.9408 &              3.7534E-04                                           & 1.9837 \\
 $\frac{\sqrt{2}}{64}$&             4.1461E-07&                    2.8278 &              9.4302E-05                                           & 1.9928 \\
\hline
  \end{tabular}
  \end{center}
\end{table}
\begin{table}[!htbp]
\caption{WG error convergence orders for $p$ with $\lambda=1$ and $\tau=h^2$ on uniform triangular meshes}\label{table2}
  \begin{center}
  \begin{tabular}{ccccc}
   \hline
 $h$&   $\|Q_0p(t_n)-p_0^n\|$&        Order  &        $|\!|\!|Q_hp(t_n)-p_h^n|\!|\!|_{W}$  &    Order \\ \hline
$\frac{\sqrt{2}}{2}$&      4.5721E-03&                   -    &                1.9844E-02                   &     - \\
$\frac{\sqrt{2}}{4}$&      8.9546E-04&                 2.3521 &                7.2165E-03                   &   1.4593 \\
$\frac{\sqrt{2}}{8}$&      2.2088E-04&                 2.0194 &                3.4149E-03                   &   1.0794  \\
$\frac{\sqrt{2}}{16}$&     5.5267E-05&                 1.9987 &                1.6864E-03                   &   1.0180 \\
$\frac{\sqrt{2}}{32}$&     1.3825E-05&                 1.9992 &                8.4064E-04                   &   1.0044 \\
$\frac{\sqrt{2}}{64}$&     3.4568E-06&                 1.9998 &                4.2001E-04                   &   1.0011 \\
\hline
  \end{tabular}
  \end{center}
\end{table}

\subsubsection{Rectangular meshes}

\indent In this test, we make use of a uniform rectangular mesh $\mathcal{T}_h$ with $h=\frac{1}{2}, \frac{1}{4}, \frac{1}{8}, \frac{1}{16}, \frac{1}{32}, \frac{1}{64}$ and the time step $\tau=h^2$, and the convergence rates are depicted in Table \ref{table3} and \ref{table4}. From the two tables, it can be seen that the four norms for $\bm{u}$ and $p$ all achieve the optimal error convergence orders which are accordance with our theoretical analysis.
\begin{table}[!htbp]
\caption{WG error convergence orders for $\bm{u}$ with $\lambda=1$ and $\tau=h^2$ on uniform rectangular meshes}\label{table3}
  \begin{center}
  \begin{tabular}{ccccc}
   \hline
 $h$&    $\|\bm{Q}_0\bm{u}(t_n)-\bm{u}_0^n\|$     &Order         & $|\!|\!|\bm{Q}_h\bm{u}(t_n)-\bm{u}_h^n|\!|\!|_{\bm{V}}$   & Order \\ \hline
$\frac{1}{2}$&        2.1907E-02&                  - &                            1.2383E-01&                       -  \\
$\frac{1}{4}$&        3.3989E-03&              2.6883&                            3.9438E-02&                    1.6507 \\
$\frac{1}{8}$&        4.5465E-04&              2.9022&                            1.1434E-02&                    1.7863 \\
$\frac{1}{16}$&       5.6820E-05&              3.0003&                            3.1454E-03&                    1.8620 \\
$\frac{1}{32}$&       7.1510E-06&              2.9902&                            8.3765E-04&                    1.9088 \\
$\frac{1}{64}$&       9.8402E-07&              2.8614&                            2.1730E-04&                    1.9467 \\
\hline
  \end{tabular}
  \end{center}
\end{table}
\begin{table}[!htbp]
\caption{WG error convergence orders for $p$ with $\lambda=1$ and $\tau=h^2$ on uniform rectangular meshes}\label{table4}
  \begin{center}
  \begin{tabular}{ccccc}
   \hline
 $h$&   $\|Q_0p(t_n)-p_0^n\|$&        Order  &        $|\!|\!|Q_hp(t_n)-p_h^n|\!|\!|_{W}$  &    Order \\ \hline
$\frac{1}{2}$&     8.3611E-03&        -&                     2.9900E-02&              - \\
$\frac{1}{4}$&     1.8531E-03&        2.1737&                1.1653E-02&              1.3594 \\
$\frac{1}{8}$&     4.7047E-04&        1.9778&                5.5779E-03&              1.0630 \\
$\frac{1}{16}$&    1.1930E-04&        1.9795&                2.7740E-03&              1.0077 \\
$\frac{1}{32}$&    2.9973E-05&        1.9929&                1.3862E-03&              1.0008 \\
$\frac{1}{64}$&    7.5036E-06&        1.9980&                6.9304E-04&              1.0001 \\
\hline
  \end{tabular}
  \end{center}
\end{table}

\subsubsection{Hybrid polygonal meshes}

\indent In this subsection, we partition the two-dimensional domain $\Omega=(0,1)^2$ into hybrid polygonal meshes which are shown as Figure \ref{polymesh}, where $N_h$ is the number of complete subdivisions on each boundary of $\Gamma$. Numerical tests are conducted with the time step $\tau = N_h^{-2}$. Table \ref{table5} and \ref{table6} render all errors and convergence results of optimal orders, which agree with our expectation.
\begin{figure}[!hbt]
\centering
\includegraphics[width=5.4cm,height=4cm]{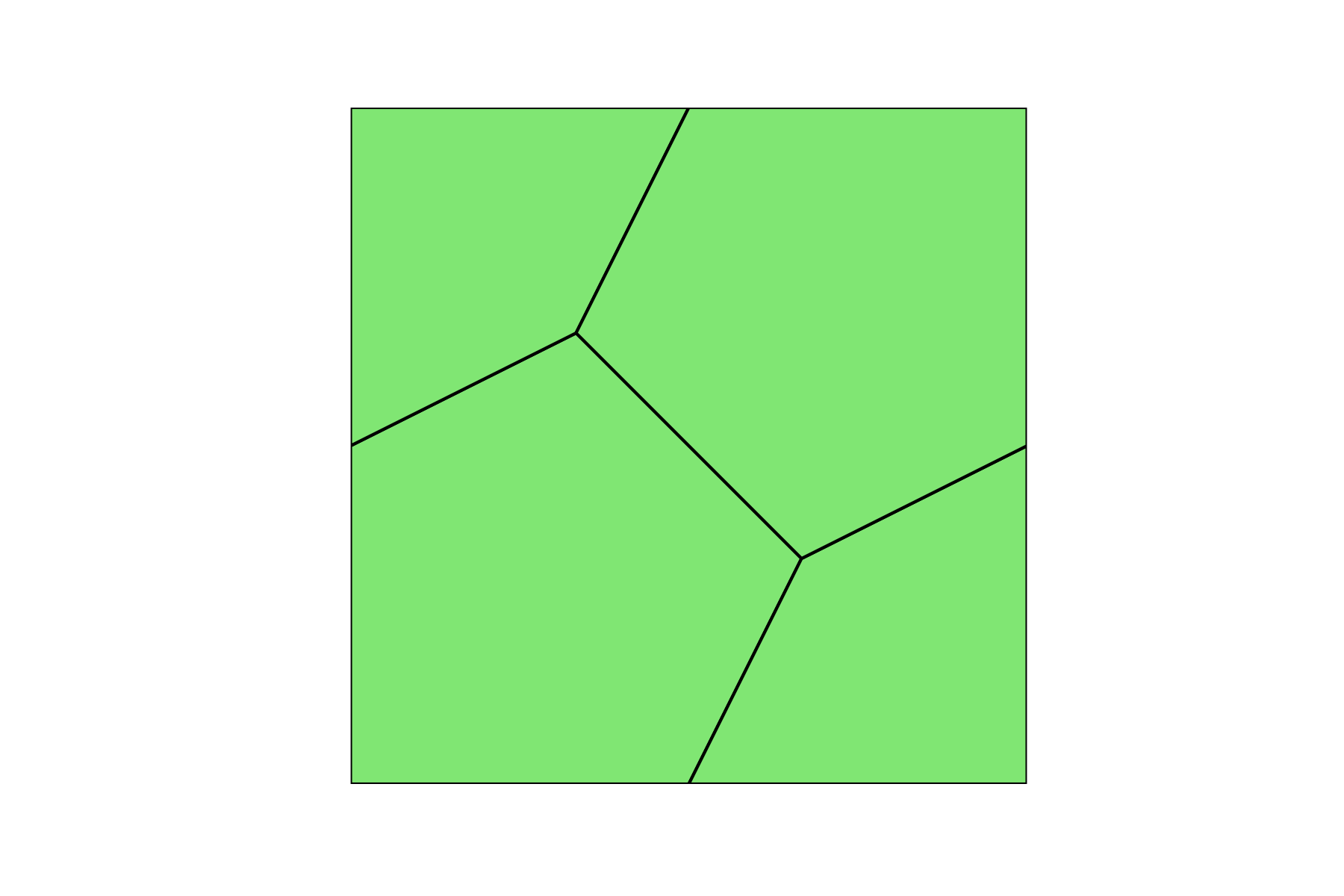}
\includegraphics[width=5.4cm,height=4cm]{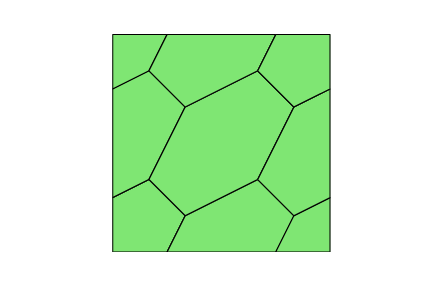}
\includegraphics[width=5.4cm,height=4cm]{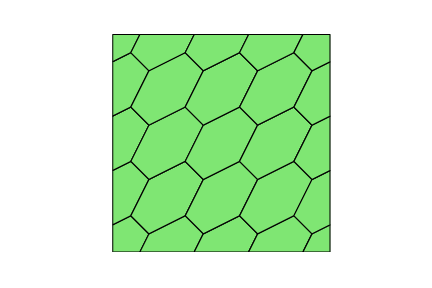}
\caption{Polygonal grids with $N_h=1, 2, 4$.}\label{polymesh}
\end{figure}
\begin{table}[!htbp]
\caption{WG error convergence orders for $\bm{u}$ with $\lambda=1$ and $\tau=N_h^{-2}$ on hybrid polygonal meshes}\label{table5}
  \begin{center}
  \begin{tabular}{ccccc}
   \hline
$N_h$&    $\|\bm{Q}_0\bm{u}(t_n)-\bm{u}_0^n\|$     &Order         & $|\!|\!|\bm{Q}_h\bm{u}(t_n)-\bm{u}_h^n|\!|\!|_{\bm{V}}$   & Order \\ \hline
2&        1.3017E-02&                   -           & 8.1600E-02                   & - \\
4&        2.6812E-03&                3.1095         & 3.2019E-02                   & 1.8411 \\
8&        4.3683E-04&                3.1545         & 1.0204E-02                   & 1.9882 \\
16&       6.1513E-05&                3.1400         & 2.8763E-03                   & 2.0282 \\
32&       8.2731E-06&                3.0607         & 7.6434E-04                   & 2.0218 \\
64&       1.1501E-06&                2.9303         & 1.9708E-04                   & 2.0129 \\
\hline
  \end{tabular}
  \end{center}
\end{table}
\begin{table}[!htbp]
\caption{WG error convergence orders for $p$ with $\lambda=1$ and $\tau=N_h^{-2}$ on hybrid polygonal meshes}\label{table6}
  \begin{center}
  \begin{tabular}{ccccc}
   \hline
$N_h$&   $\|Q_0p(t_n)-p_0^n\|$&        Order  &        $|\!|\!|Q_hp(t_n)-p_h^n|\!|\!|_{W}$  &    Order \\ \hline
2&     3.5652E-03&        -  &                1.7364E-02&          - \\	
4&     1.3776E-03&        1.8715&             9.5444E-03&          1.1778 \\
8&     4.1133E-04&        2.1013 &            5.0863E-03&          1.0942 \\
16&    1.1304E-04&        2.0690 &            2.6650E-03&          1.0353 \\
32&    2.9658E-05&        2.0412 &            1.3699E-03&          1.0152 \\
64&    7.5972E-06&        2.0226 &            6.9530E-04&          1.0071 \\
\hline
  \end{tabular}
  \end{center}
\end{table}

\subsection{Tests for locking-free when $\lambda\to\infty$}

\indent The aim of this subsection is to validate the locking-free property of our WG method. The right-hand side terms are chosen so that the exact solution is
\begin{align*}
\displaystyle\bm{u} &= \binom{\exp(-t)(\sin(2\pi y)(-1+\cos(2\pi x))+\frac{1}{\mu+\lambda}\sin(\pi x)\sin(\pi y))}{\exp(-t)(\sin(2\pi x)(1-\cos(2\pi y))+\frac{1}{\mu+\lambda}\sin(\pi x)\sin(\pi y))},\\
p &= \exp(-t)\sin(\pi x)\sin(\pi y).
\end{align*}
For the verification of locking-free of our method, we compare the WG method with the lowest order Taylor-Hood element, i.e., $[P_2]^2\times P_1$ element, for $\bm{u}$ and $p$ with the three choices of $\lambda=1, 10^4, 10^8$. We denote the finite element solution of $\bm{u}$ and $p$ by $\bm{u}_{hfem}^n$ and $p_{hfem}^n$, respectively. The error convergence rates on uniform triangulation are computed with the time step $\tau=h^2$.

Figure \ref{locking1}-\ref{locking8} depict the errors of our method and the conforming finite element method for the different $\lambda$. As can be observed from these figures, the locking problem does not affect the finite element approximation to $p$, but only the one to $\bm{u}$. As $\lambda$ goes to infinity, the orders of error convergence of finite element solution $\bm{u}_{hfem}^n$ degenerate from the optimal ones to the lower ones. While WG method has no significant fluctuations and keeps the optimal convergence orders during the entire change in $\lambda$, which clearly shows the advantage of our method.
\begin{figure}[!hbt]
\centering
\includegraphics[width=8.1cm,height=8cm]{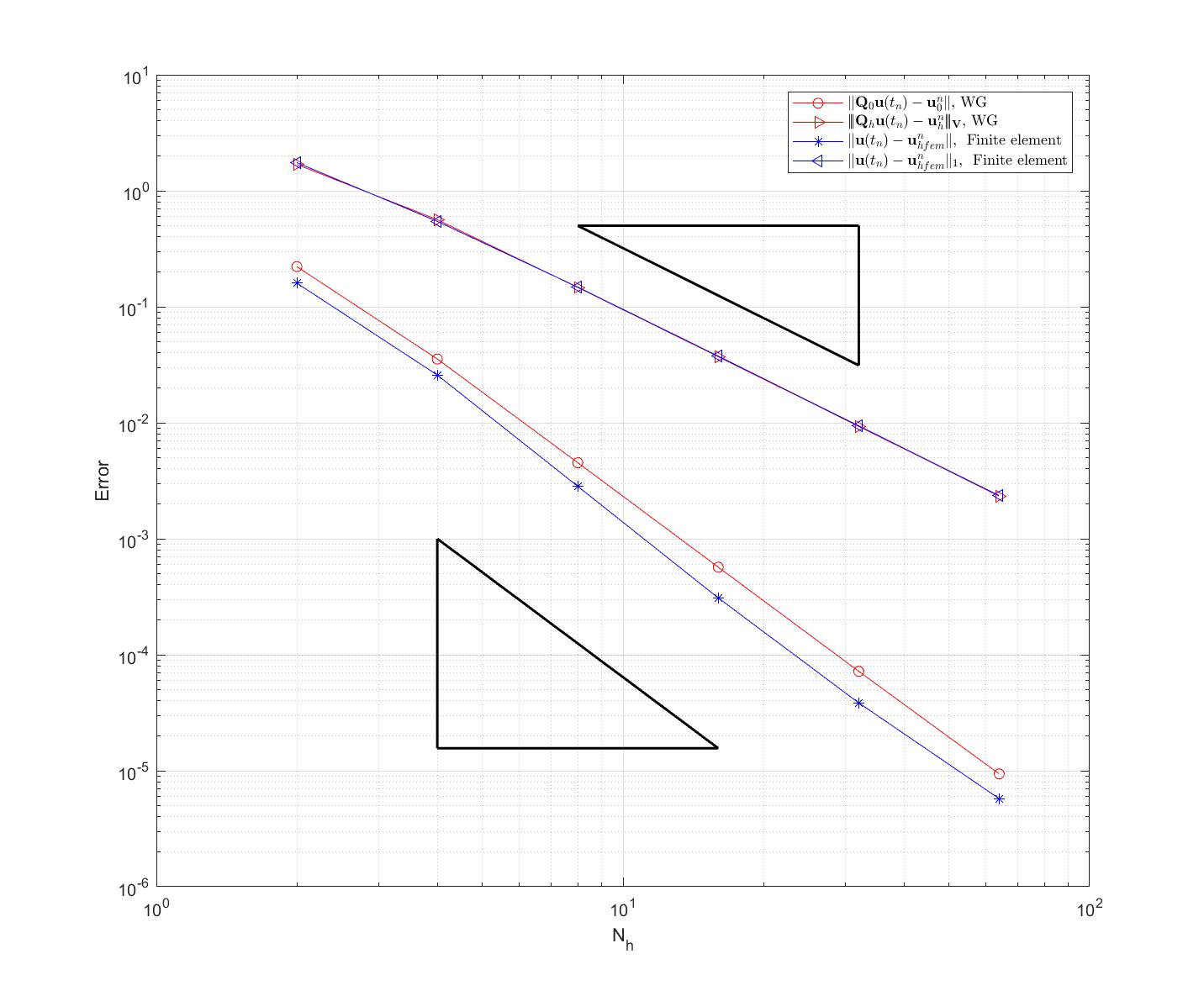}
\includegraphics[width=8.1cm,height=8cm]{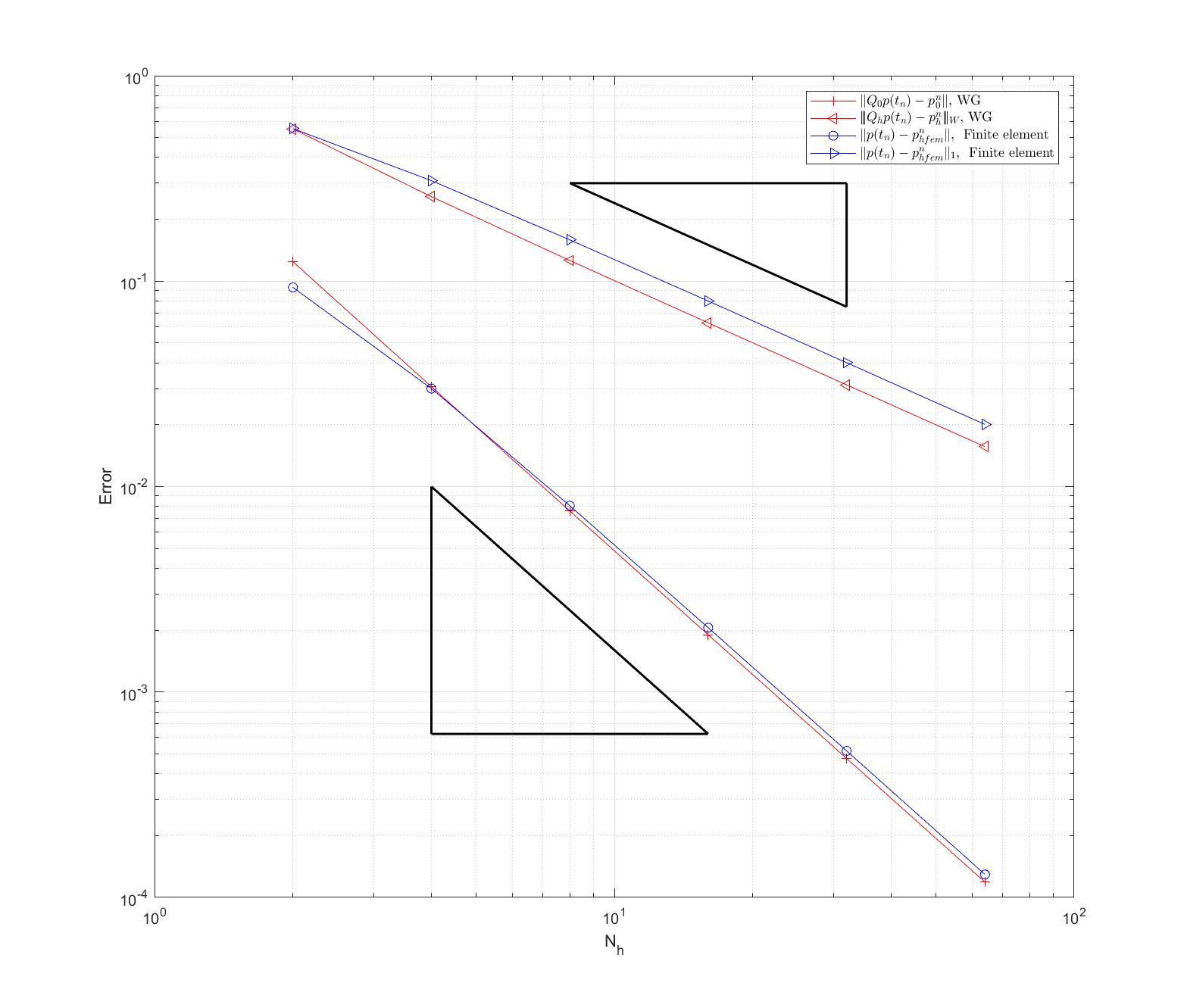}
\caption{Errors for $\bm{u}$ (Left) and $p$ (Right) with $\lambda=1$.}\label{locking1}
\end{figure}
\begin{figure}[!hbt]
\centering
\includegraphics[width=8.1cm,height=8cm]{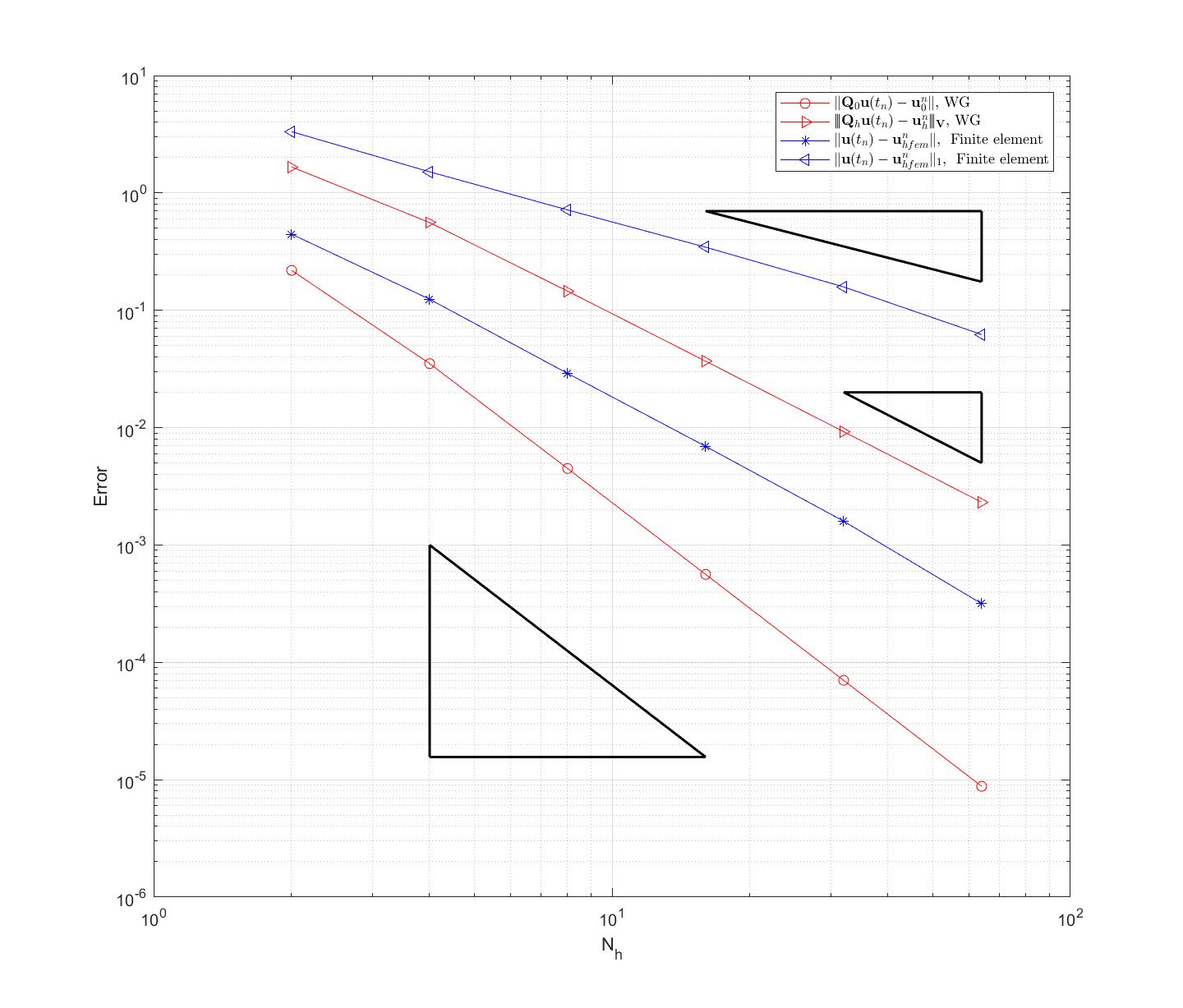}
\includegraphics[width=8.1cm,height=8cm]{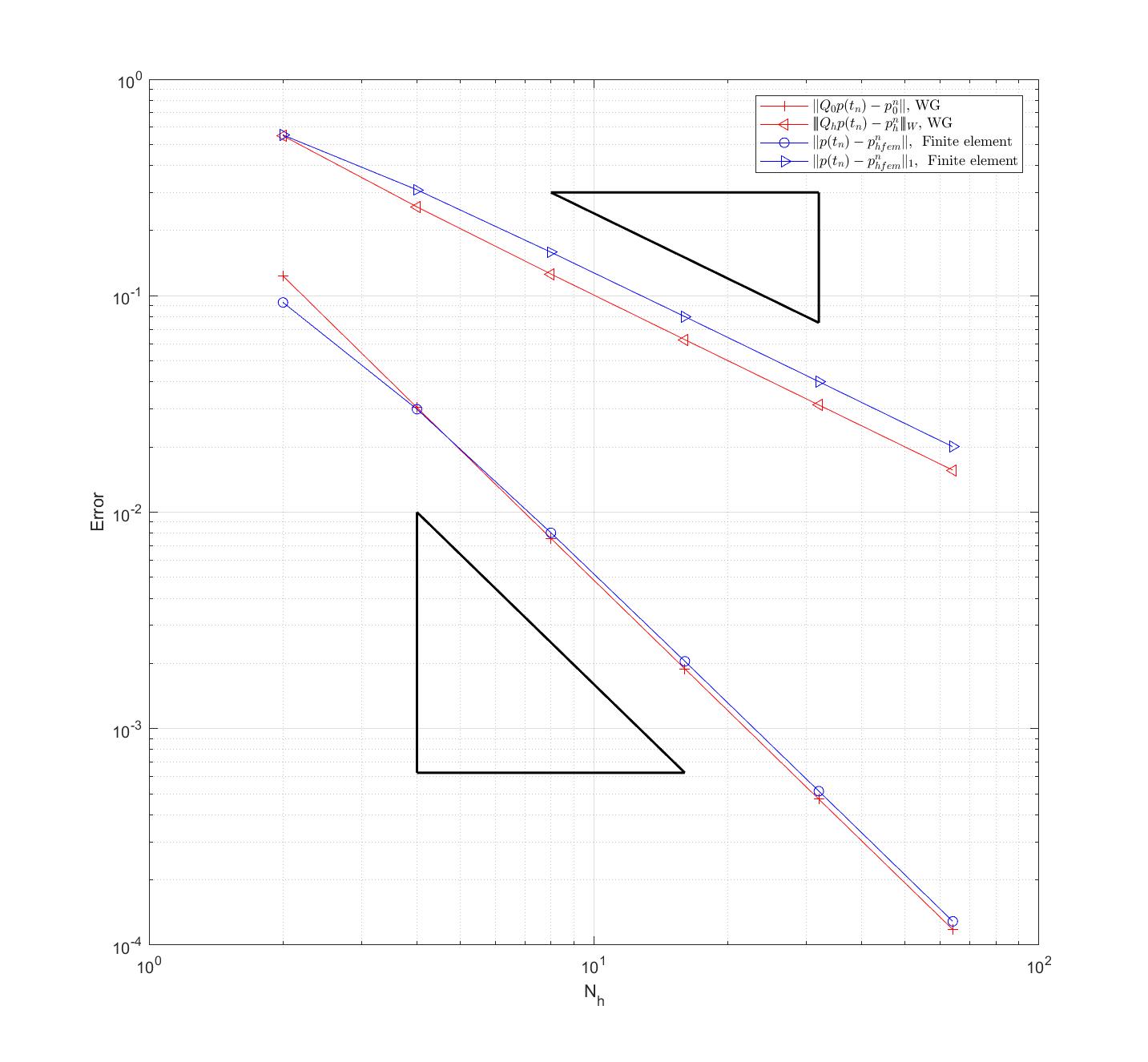}
\caption{Errors for $\bm{u}$ (Left) and $p$ (Right) with $\lambda=10^4$.}\label{locking4}
\end{figure}
\begin{figure}[!hbt]
\centering
\includegraphics[width=8.1cm,height=8cm]{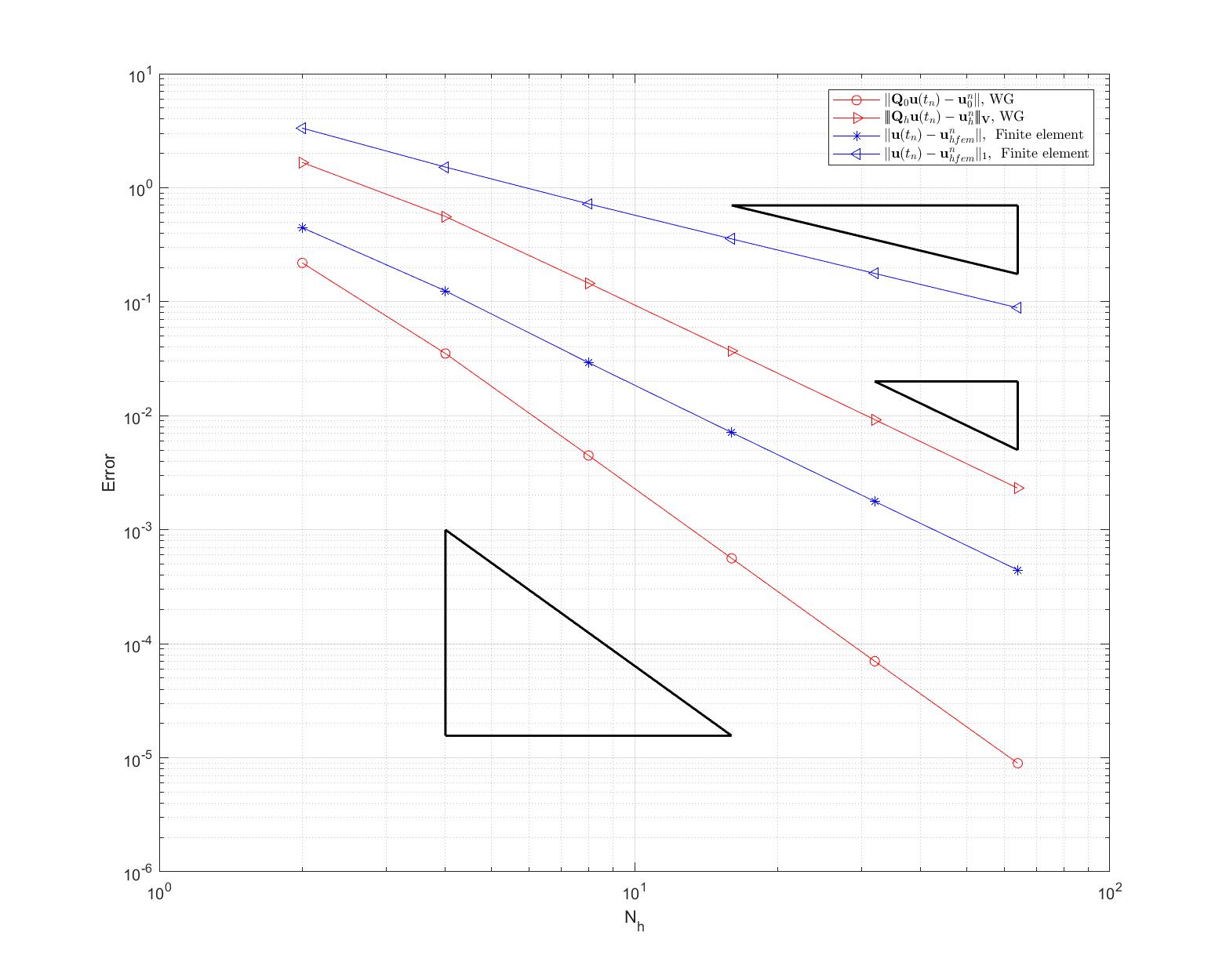}
\includegraphics[width=8.1cm,height=8cm]{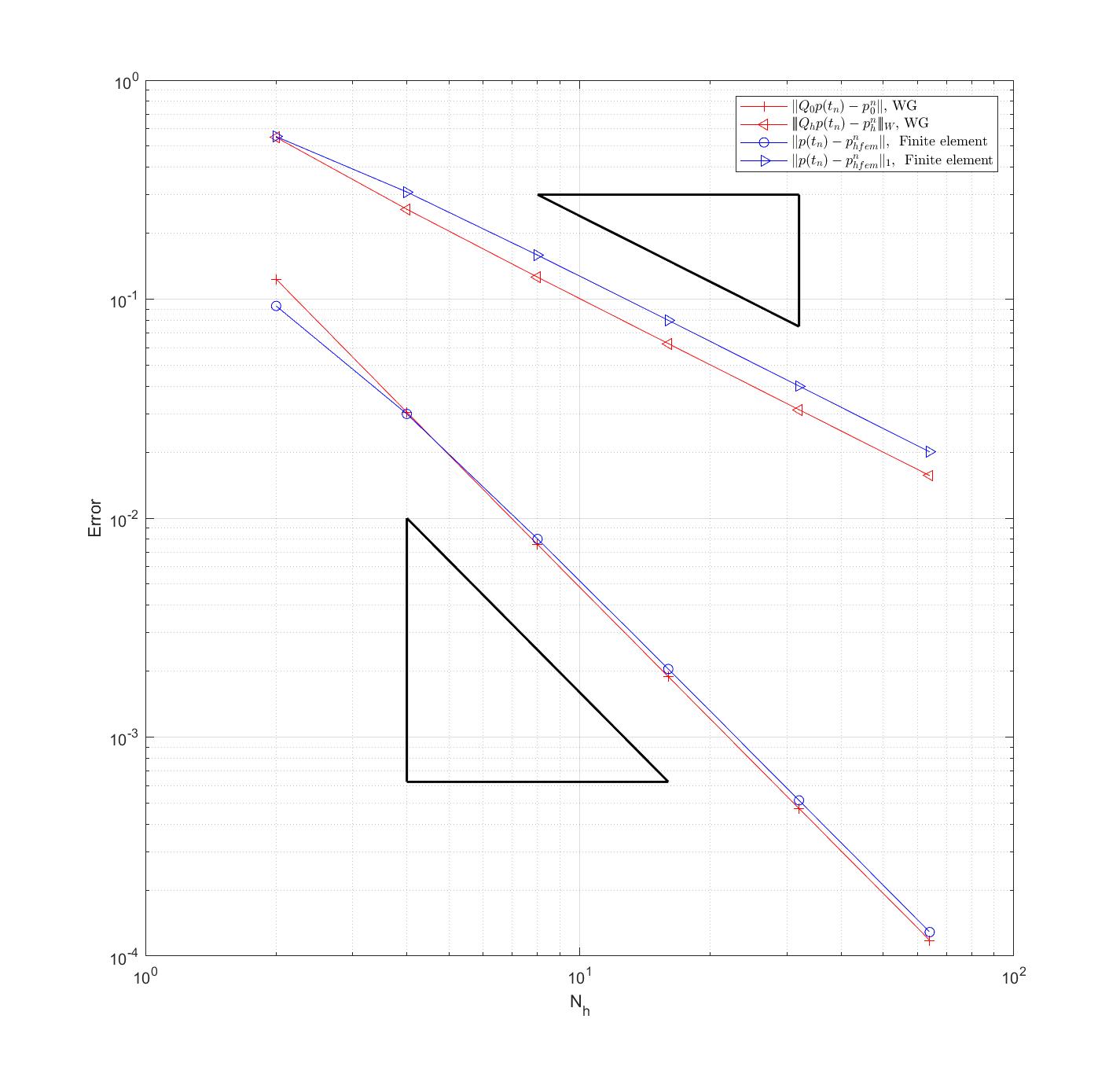}
\caption{Errors for $\bm{u}$ (Left) and $p$ (Right) with $\lambda=10^8$.}\label{locking8}
\end{figure}

\section*{Acknowledgments}
\indent This work was sponsored by the Research Foundation for Beijing University of Technology New Faculty Grant No. 006000514122516. The computations here were partly done on the high performance computers of State Key Laboratory of Scientific and Engineering Computing, Chinese Academy of Sciences. The authors sincerely thank Dr. Hui Peng for the discussion about mathematical models and their solving algorithms. \\

\section*{References}
\bibliographystyle{elsarticle-harv}
\bibliography{mybibfile}

\end{document}